\documentclass{amsart}[11pt]
 \usepackage{amsmath}
\usepackage{amsfonts}
\usepackage{amssymb}
\usepackage{float}
\usepackage{url}
\usepackage{graphicx}
\usepackage{fullpage}
\usepackage{color}
\usepackage{epsfig}

\newtheorem{lemma}{Lemma}%[section]
\newtheorem{theorem}{Theorem}
\newtheorem{example}{Example}[section]
\newtheorem{remark}{Remark}[section]

\begin{document}

\title{\bf  Single cut and multicut SDDP with cut selection for multistage stochastic linear programs: convergence proof and numerical experiments}

\maketitle

\vspace*{0.5cm}

\begin{center}
\begin{tabular}{cc}
\begin{tabular}{c}
Vincent Guigues (corresponding author)\\
School of Applied Mathematics, FGV\\
Praia de Botafogo, Rio de Janeiro, Brazil\\ 
{\tt vguigues@fgv.br}\\
\end{tabular}
\begin{tabular}{c}
Michelle Bandarra\\
School of Applied Mathematics, FGV\\
Praia de Botafogo, Rio de Janeiro, Brazil\\ 
{\tt michelle.bandarra@mirow.com.br}\\
\end{tabular}
\end{tabular}
\end{center}

\vspace*{0.5cm}

%\author{Michelle Bandarra$^{\rm a}$ \and Vincent Guigues$^{\rm a}$}

%\thanks{$^\ast$Corresponding author: Vincent Guigues, email: {\tt vincent.guigues@gmail.com}.}

%\date{Version from \today }%Received: date / Accepted: date}
% The correct dates will be entered by the editor

% \begin{center}
% $^{\rm A}$ School of Applied Mathematics FGV/EMAp,\\
% 190 Praia de Botafogo\\
% 22 250-900 Rio de Janeiro, Brazil
% \end{center}

\date{}

\begin{abstract}
We introduce a variant of Multicut Decomposition Algorithms (MuDA), called CuSMuDA (Cut Selection for Multicut Decomposition Algorithms), for solving 
multistage stochastic linear programs that incorporates a class of  cut selection
strategies to choose the most relevant cuts of the approximate recourse functions.
This class contains Level 1 \cite{dpcuts0} and Limited Memory Level 1 \cite{guiguesejor2017} cut selection strategies, initially
introduced for respectively Stochastic Dual Dynamic Programming (SDDP) and Dual Dynamic Programming (DDP).
We prove the almost sure convergence of the method in a finite number
of iterations and obtain as a by-product the almost sure convergence in a finite number of iterations of SDDP combined with our class of cut selection strategies.

We compare the performance of MuDA, SDDP, and their variants with cut selection (using Level 1 and Limited Memory Level 1) on
several instances of a portfolio problem and of an inventory problem.
On these experiments, in general, SDDP is quicker (i.e., satisfies the stopping criterion quicker) than 
MuDA and cut selection allows us to decrease the computational bulk with Limited Memory Level 1 being more efficient (sometimes
much more) than Level 1.\\
\end{abstract}



\section{Introduction}

Multistage stochastic optimization problems are common in many areas of engineering and in finance.
However, solving these problems is challenging and in general requires decomposition techniques.
Two popular decomposition methods are Approximate Dynamic Programming (ADP) \cite{powellbook}
and sampling-based variants of the Nested Decomposition (ND) algorithm \cite{birgemulti, birge-louv-book}
and of the Multicut Nested Decomposition (MuND) algorithm \cite{birgelouv, gassmanmp90}.
The introduction of sampling within ND was proposed by \cite{pereira} and the corresponding
method is usually called Stochastic Dual Dynamic Programming (SDDP).
Several enhancements and extensions of SDDP have been proposed such as 
CUPPS \cite{chenpowell99}, ReSA \cite{resa}, AND \cite{birgedono}, DOASA \cite{philpot},
risk-averse variants in \cite{guiguesrom10, guiguesrom12, philpmatos, guiguescoap2013, shaptekaya, kozmikmorton},
cut formulas for nonlinear problems in \cite{guiguessiopt2016}, regularizations
in \cite{powellasamov} for linear problems
and SDDP-REG in \cite{guilejtekregsddp} for nonlinear problems.  
Convergence of these methods was proved in \cite{philpot} for linear problems,
in \cite{lecphilgirar12} for nonlinear risk-neutral problems, and in \cite{guiguessiopt2016} for risk-averse nonlinear problems. 
All these algorithms compute lower approximations of the cost-to-go functions expressed as a supremum of affine functions called
optimality cuts. Typically, at each iteration, a fixed  number of cuts is added for each cost-to-go function.
Therefore, techniques to reduce the number of cuts in each subproblem (referred to as cut selection or cut pruning)
may be helpful to speed up the convergence  of these methods.
In stochastic optimization, the problem of cut selection for lower approximations of the cost-to-go functions
associated to each node of the scenario tree was discussed for the first time in \cite{rusparallel93} where only
the active cuts are selected.
Pruning strategies of basis (quadratic) functions have been proposed in \cite{gaubenezheng} and 
\cite{mcdesgaube} for max-plus based approximation methods which, similarly to SDDP, approximate the cost-to-go functions of a nonlinear
optimal control problem by a supremum of basis functions. More precisely, in 
\cite{gaubenezheng}, a fixed number of cuts is eliminated and cut selection is done solving a combinatorial optimization problem.
For SDDP, in \cite{shaptekaya} it is suggested at some iterations to eliminate redundant cuts
(a cut is redundant if it is never active in describing the lower approximate cost-to-go function).
This procedure is called test of usefulness in \cite{pfeifferetalcuts}.
This requires solving at each stage as many linear programs as there are cuts. 
In \cite{pfeifferetalcuts} and \cite{dpcuts0}, only the cuts that have the largest value for at least
one of the trial points computed are considered relevant. 
This 
strategy is called the Territory algorithm in \cite{pfeifferetalcuts}
and Level 1 cut selection in \cite{dpcuts0}. It was presented for the first time
in 2007 at the ROADEF congress by David Game and Guillaume Le Roy (GDF-Suez), see \cite{pfeifferetalcuts}.

Sampling can also be incorporated into  MuND algorithm (see for instance \cite{zhang2016}).
This algorithm builds many more cuts  than SDDP per iteration and therefore each iteration takes
more time but less iterations are in general needed to satisfy some stopping criterion.
Therefore, cut selection strategies could also be useful. However, to the best of our knowledge,
the combination of multicut decomposition methods 
with cut selection strategies, refereed to as CuSMuDA (Cut Selection for Multicut Decomposition Algorithms) in the sequel, has not been proposed so far. 
In this context, the objectives and contributions  of this paper are the following:
\begin{itemize}
\item[(A)] we propose  cut selection strategies that are more efficient than the 
aforementioned ones. More precisely, instead of selecting all the cuts that are the highest at the trial points, we introduce a set of selectors
that select some subset of these cuts. The selectors have to satisfy an assumption (Assumption (H3), see Section \ref{sec:cutselectionalgo})
to ensure the convergence of SDDP and MuDA combined with these cut selection strategies. 
We obtain a family of cut selection strategies; a given strategy corresponding to a choice of selectors along the iterations.
In this family, the most economic (in terms of memory) cut selection strategy satisfying (H3)
is the  Limited Memory Level 1 (LML 1) strategy  which selects at each trial point only one cut, namely the oldest cut. 
This strategy was introduced in 
\cite{guiguesejor2017} in the context of Dual Dynamic Programming but can be straightforwardly applied to SDDP and MuDA. 
The least economic strategy, i.e., the one that keeps the largest amount of cuts, is Level 1. ``Between" these two strategies, using the flexibility offered 
by the selectors (as long as Assumption (H3) is satisfied by these selectors), we obtain
a (large) family of cut selection strategies.
\item[(B)] We introduce and describe CuSMuDA, a combination of MuDA with our family of cut selection strategies.
\item[(C)] We prove the almost sure convergence of CuSMuDA in a finite number of iterations.
This proof extends the theory in \cite{guiguesejor2017} in two aspects:
(i) first the stochastic case is considered, i.e., SDDP and multicut SDDP are considered whereas the deterministic case, i.e., DDP, was considered in \cite{guiguesejor2017} and second
(ii) more general cut selection strategies are dealt with. Item (ii) requires an additional technical discussion, see Lemma \ref{lemmafinitecuts}.
\item[(D)] We present the results of numerical experiments comparing the performance of six solution methods 
on several instances of a portfolio problem and of an inventory problem.
These six solution methods are SDDP, SDDP with Level 1 cut selection, SDDP with LML 1 cut selection, MuDA,
CuSMuDA with Level 1 cut selection, and CuSMuDA with LML 1 cut selection.
To the best of our knowledge, these are the first numerical experiments on SDDP with LML 1 cut selection
and the first experiments on multicut SDDP combined with cut selection.
The main conclusions of these experiments are the following: 
\begin{itemize}
\item in general, for a given instance, CuSMuDA with LML 1 cut selection (resp. SDDP combined with LML 1
cut selection) is more efficient (i.e., allows us to satisfy
more quickly the stopping criterion) than CuSMuDA with Level 1 cut selection (resp. SDDP combined with 
Level 1 cut selection), itself much more efficient than MuDA (resp. SDDP).
Typically, variants with cut selection require more iterations but iterations are quicker
with very few cuts selected for the first stages. However, on some instances, CuSMuDA with Level 1
cut selection still selected a large proportion of cuts for all stages and was less efficient than
both MuDA and CuSMuDA with LML 1 cut selection.
\item MuDA (resp. CuSMuDA) in general requires much more computational bulk than SDDP (resp. SDDP with cut selection).
However, on some instances, CuSMuDA with LML 1 cut selection and SDDP combined with LML 1 cut selection
have shown similar performances.
We also expect CuSMuDA to be more efficient than SDDP with cut selection 
when MuDA is already more efficient than SDDP.
Even if this is not often the case, the results of numerical experiments
on several instances of multistage stochastic linear programs where MuDA is quicker than SDDP are
reported in \cite{birge1996}.
\end{itemize}
\end{itemize}

The outline of the study is as follows. The class of problems considered and assumptions are discussed in Section \ref{classassumptions}.
In Subsection \ref{msda} we recall sampling-based MuND while in Subsection \ref{sec:cutselectionalgo} CuSMuDA is described.
In Section \ref{sec:convanalysis}, we  prove Theorem \ref{convproof} which states that CuSMuDA converges almost surely in a finite number of iterations
to an optimal policy.
As a by-product, we obtain the almost sure convergence of SDDP combined with our class of cut selection strategies (in particular Level 1 
and LML 1) in a finite number of iterations.
Finally, numerical experiments are presented in Section \ref{sec:numsim}.

Throughout the paper, the usual scalar product in $\mathbb{R}^n$ is denoted by $\langle x, y\rangle = x^\top y$ for $x, y \in \mathbb{R}^n$.
% \cite{birge-louv-book,birgemulti}
% 
% 
% \cite{resa}
% 
% \cite{powellbook}  \cite{birge-louv-book,birgemulti}
% \cite{pereira}.
% \cite{shapsddp}, \cite{philpmatos}, \cite{guiguesrom10}, \cite{guiguesrom12}, \cite{guiguescoap2013},
% \cite{kozmikmorton}, \cite{shaptekaya2}, \cite{dpcuts0}. 
% \cite{rusparallel93} \cite{gaubenezheng} \cite{mcdesgaube}
%  \cite{gaubenezheng},
%  \cite{shaptekaya2}  .
% \cite{gassmanmp90}
%  \cite{pfeifferetalcuts}   \ref{algorithmddp}  
% 
% \cite{chenpowell99}
% 
% {\cite{BEHL, bentalmargnem00}
% 
% \cite{guilejtekregsddp}
% \cite{shaptekaya}
% \cite{shapsddp}
% \cite{kozmikmorton}
% \cite{philpot}
% \cite{morton}
% \cite{guiguesrom12}
% \cite{guiguescoap2013}
% \cite{guiguesrom10}
% \cite{guiguessiopt2016}
% 
% \cite{philpmatos2}
% 
% 
% \cite{guiguesejor2017}
% \cite{zhang2016}
% \cite{lecphilgirar12}
%\cite{birgedono}

\section{Problem formulation and assumptions}\label{classassumptions}

We are interested in solution methods for linear stochastic dynamic programming equations:
the first stage problem is 
\begin{equation}\label{firststodp}
\mathcal{Q}_1( x_0 ) = \left\{
\begin{array}{l}
\inf_{x_1 \in \mathbb{R}^n} \langle  c_1 ,  x_1 \rangle + \mathcal{Q}_2 ( x_1 )\\
A_{1} x_{1} + B_{1} x_{0} = b_1,
x_1 \geq 0
\end{array}
\right.
\end{equation}
for $x_0$ given and for $t=2,\ldots,T$, $\mathcal{Q}_t( x_{t-1} )= \mathbb{E}_{\xi_t}[ \mathfrak{Q}_t ( x_{t-1}, \xi_{t}  )  ]$ with
\begin{equation}\label{secondstodp}
\mathfrak{Q}_t ( x_{t-1}, \xi_{t}  ) = 
\left\{ 
\begin{array}{l}
\inf_{x_t \in \mathbb{R}^n} \langle c_t ,  x_t \rangle + \mathcal{Q}_{t+1} ( x_t )\\
A_{t} x_{t} + B_{t} x_{t-1} = b_t,
x_t \geq 0,
\end{array}
\right.
\end{equation}
with the convention that $\mathcal{Q}_{T+1}$ is null and
where for $t=2,\ldots,T$, random vector $\xi_t$ corresponds to the concatenation of the elements in random matrices $A_t, B_t$ which have a known
finite number of rows and random vectors $b_t, c_t$
(it is assumed that $\xi_1$ is not random). For convenience, we will denote 
$$
X_t(x_{t-1}, \xi_t):=\{x_t \in \mathbb{R}^n : A_{t} x_{t} + B_{t} x_{t-1} = b_t, \,x_t \geq 0 \}.
$$
We make the following assumptions:
\begin{itemize}
\item[(H1)] The random vectors $\xi_2, \ldots, \xi_T$ are independent and have discrete distributions with finite support. 
\item[(H2)] The set $X_1(x_{0}, \xi_1 )$ is nonempty and bounded and for every $x_1 \in X_1(x_{0}, \xi_1 )$,
for every $t=2,\ldots,T$, for every realization $\tilde \xi_2, \ldots, \tilde\xi_t$ of $\xi_2,\ldots,\xi_t$,
for every $x_{\tau} \in X_{\tau}( x_{\tau-1} , \tilde \xi_{\tau}), \tau=2,\ldots,t-1$, the set $X_t( x_{t-1} , {\tilde \xi}_t )$
is nonempty and bounded.
\end{itemize}
We will denote by $\Theta_t = \{\xi_{t 1},\ldots,\xi_{t M_t} \}$ the support  of $\xi_t$ for stage $t$ with
$p_{t i}= \mathbb{P}(\xi_t = \xi_{t i}) >0, i=1,\ldots,M_t$ and with vector $\xi_{t j}$ being the concatenation
of the elements in $A_{t j}, B_{t j}, b_{t j}, c_{t j}$.

\section{Algorithms}

\subsection{Multicut stochastic decomposition}\label{msda}

The multicut stochastic  decomposition method approximates the function  $\mathfrak{Q}_t(\cdot, \xi_{t j})$ at iteration $k$ for $t=2,\ldots,T$, $j=1,\ldots,M_t$, 
 by a piecewise affine lower bounding function $\mathfrak{Q}_t^k(\cdot, \xi_{t j})$ 
which is a maximum of  $k$ affine functions $\mathcal{C}_{t j}^i$ called cuts:
$$
\mathfrak{Q}_t^k(  x_{t-1}, \xi_{t j}) = \max_{1 \leq i \leq k} \mathcal{C}_{t j}^i( x_{t-1}  ) \mbox{ with }\mathcal{C}_{t j}^i (x_{t-1})=\theta_{t j}^i + \langle \beta_{t j}^i , x_{t-1} \rangle
$$
where coefficients $\theta_{t j}^i, \beta_{t j}^i$ are computed as explained below. These approximations
provide the lower bounding functions
\begin{equation}\label{approxQt}
\mathcal{Q}_{t}^k ( x_{t-1} )=  \sum_{j=1}^{M_t} p_{t j} \mathfrak{Q}_{t}^k (x_{t-1} , \xi_{t j})
\end{equation}
for $\mathcal{Q}_{t}$. Since $\mathcal{Q}_{T+1}$ is the null function, we will also define
$\mathcal{Q}_{T+1}^k \equiv 0$. The steps of MuDA are described below.\\

\par {\textbf{Step 1: Initialization.}} For $t=2,\ldots,T$, $j=1,\ldots,M_t$, take for $\mathfrak{Q}_t^0(\cdot, \xi_{t j})$ 
a known lower bounding affine function for  $\mathfrak{Q}_t(\cdot, \xi_{t j})$.
Set the iteration count $k$ to 1 and $\mathcal{Q}_{T+1}^0 \equiv 0$.\\
\par {\textbf{Step 2: Forward pass.}} We generate a sample
${\tilde \xi}^k = (\tilde \xi_1^k, \tilde \xi_2^k,\ldots,\tilde \xi_T^k)$ from the distribution of $(\xi_1,\xi_2,\ldots,\xi_T)$,
with the convention that $\tilde \xi_1^k = \xi_1$ (here and in what follows, the tilde symbol will be used to represent realizations of random variables:
for random variable $\xi$, $\tilde \xi$ is a realization of $\xi$). Using approximation $\mathfrak{Q}_t^{k-1}(\cdot, \xi_{t j})$
of $\mathfrak{Q}_t(\cdot, \xi_{t j})$  (computed at previous iterations), we solve the problem
\begin{equation}\label{pbforwardpass}
\left\{
\begin{array}{l}
\inf_{x_t \in \mathbb{R}^n} \langle  x_t ,  {\tilde c}_t^k \rangle  + \mathcal{Q}_{t+1}^{k-1} ( x_t )\\
x_t \in X_t(x_{t-1}^k, {\tilde \xi}_{t}^k )
\end{array}
\right.
\end{equation}
for $t=1,\ldots,T$,
where $x_0^k=x_0$ and  $\mathcal{Q}_{t+1}^{k-1}$ is given by \eqref{approxQt} with $k$ replaced by $k-1$. Let $x_t^k$ be an optimal solution of the problem.\\
\par {\textbf{Step 3: Backward pass.}} 
The backward pass builds cuts for $\mathfrak{Q}_t(\cdot, \xi_{t j})$ at $x_{t-1}^k$ computed in the forward pass.
For $k \geq 1$ and $t=1,\ldots,T$, we introduce
the function ${\underline{\mathfrak{Q}}}_t^k : \mathbb{R}^n {\small{\times}} \Theta_t \rightarrow \mathbb{R}$ given by
\begin{equation}\label{backwardt0}
{\underline{\mathfrak{Q}}}_t^k (x_{t-1} , \xi_t   ) =  
\left\{
\begin{array}{l}
\inf_{x_t \in \mathbb{R}^n} \langle c_t ,  x_t \rangle + \mathcal{Q}_{t+1}^k ( x_t )\\
x_t \in X_t(x_{t-1}, \xi_{t} ),
\end{array}
\right.
\end{equation}
with the convention that $\Theta_1 = \{\xi_1\}$, and
we set $\mathcal{Q}_{T+1}^k \equiv 0$. For $j=1,\ldots,M_T$, we solve the problem
\begin{equation}\label{backwardT}
\mathfrak{Q}_T ( x_{T-1}^k, \xi_{T j}  ) = 
\left\{ 
\begin{array}{l}
\displaystyle \inf_{x_T \in \mathbb{R}^n} \langle  c_{T j} ,  x_T \rangle \\
A_{T j} x_{T} + B_{T j} x_{T-1}^k = b_{T j},
x_T \geq 0,
\end{array}
\right.
\mbox{ with dual }
\left\{ 
\begin{array}{l}
\sup_{\lambda} \langle \lambda ,   b_{T j} - B_{T j} x_{T-1}^k  \rangle \\
A_{T j}^\top \lambda \leq c_{T j}.
\end{array}
\right.
\end{equation}
Let $\lambda_{T j}^k$ be an optimal solution of the dual problem above. We get
$$
\mathfrak{Q}_T ( x_{T-1}, \xi_{T j}  ) \geq \langle \lambda_{T j}^{k},  b_{T j} - B_{T j} x_{T-1} \rangle
$$
and compute $\theta_{T j}^k=\langle  b_{T j} , \lambda_{T j}^{k} \rangle$ and  $\beta_{T j}^k = - B_{T j}^\top \lambda_{T j}^k$.
Then for $t=T-1$ down to $t=2$, knowing $\mathcal{Q}_{t+1}^k \leq \mathcal{Q}_{t+1}$,
we solve  the problem below for $j=1,\ldots,M_t$,
\begin{equation}\label{backwardt}
{\underline{\mathfrak{Q}}}_t^k ( x_{t-1}^k, \xi_{t j}  ) = 
\left\{ 
\begin{array}{l}
\displaystyle \inf_{x_t} \langle  c_{t j} ,  x_t \rangle + \mathcal{Q}_{t+1}^k ( x_t ) \\
x_t \in X_t( x_{t-1}^k , \xi_{t j} )
\end{array}
\right.
=
\left\{ 
\begin{array}{l}
\displaystyle \inf_{x_t, f} \langle c_{t j} ,  x_t \rangle + \sum_{\ell=1}^{M_{t+1}} p_{t+1 \ell}  f_{\ell} \\
A_{t j} x_{t} + B_{t j} x_{t-1}^k = b_{t j}, x_t \geq 0,\\
f_{\ell} \geq \theta_{t+1 \ell}^i + \langle \beta_{t+1 \ell}^i , x_t  \rangle, i=1,\ldots,k, \ell=1,\ldots,M_{t+1}.
\end{array}
\right.
\end{equation}
Observe that due to (H2) the above problem is feasible and has a finite optimal value. Therefore ${\underline{\mathfrak{Q}}}_t^k ( x_{t-1}^k, \xi_{t j}  )$
can be expressed as the optimal value of the corresponding dual problem:
\begin{equation}\label{dualpbtback}
{\underline{\mathfrak{Q}}}_t^k ( x_{t-1}^k, \xi_{t j}  ) = 
\left\{
\begin{array}{l}
\displaystyle \sup_{\lambda, \mu} \; \langle \lambda ,  b_{t j} - B_{t j} x_{t-1}^k   \rangle + \sum_{i=1}^k \sum_{\ell=1}^{M_{t+1}} \mu_{i \ell} \theta_{t+1 \ell}^i  \\
A_{t j}^\top \lambda + \sum_{i=1}^k \sum_{\ell=1}^{M_{t+1}} \mu_{i \ell} \beta_{t+1 \ell}^i \leq c_{t j},\\
p_{t+1 \ell} = \sum_{i=1}^k \mu_{i \ell},\,\ell=1,\ldots,M_{t+1},\\
\mu_{i \ell} \geq 0,\,i=1,\ldots,k,\ell=1,\ldots,M_{t+1}.
\end{array}
\right.
\end{equation}
Let $(\lambda_{t j}^k, \mu_{t j}^k )$ be an optimal solution of dual problem \eqref{dualpbtback}.
Using the fact that $\mathcal{Q}_{t+1}^k \leq \mathcal{Q}_{t+1}$, we get
$$
\mathfrak{Q}_t ( x_{t-1} , \xi_{t j}  ) \geq  {\underline{\mathfrak{Q}}}_t^k ( x_{t-1} , \xi_{t j}  ) \geq 
\langle \lambda_{t j}^k ,  b_{t j} - B_{t j} x_{t-1} \rangle + \langle \mu_{t j}^k , \theta_{t+1}^k \rangle
$$
and we compute 
$$
\theta_{t j}^k =\langle  \lambda_{t j}^k ,  b_{t j} \rangle +  \langle \mu_{t j}^k , \theta_{t+1}^k \rangle \mbox{ and }
\beta_{t j}^k =- B_{t j}^\top \lambda_{t j}^k.
$$
In these expressions, vector $\theta_{t+1}^k$ has components $\theta_{t+1 \ell}^i, \ell=1,\ldots,M_{t+1}, i=1,\ldots,k$,
arranged in the same order as components $\mu_{t j}^k(\ell, i)$ of $\mu_{t j}^k$.\\

\par {\textbf{Step 4:} Do $k \leftarrow k+1$ and go to Step 2.

\subsection{Multicut stochastic decomposition with cut selection}\label{sec:cutselectionalgo}

We now describe a variant of MuDA that stores all cut coefficients $\theta_{t j}^i, \beta_{t j}^i$, 
and trial points $x_{t-1}^i$,
but that uses a reduced set of cuts $\mathcal{C}_{t j}^i$ to approximate functions $\mathfrak{Q}_{t}(\cdot,\xi_{t j})$
when solving problem  
\eqref{pbforwardpass} in the forward pass and
\eqref{backwardt} in the backward pass.
Let $S_{t j}^k$ be the set of indices of the cuts selected at the end of iteration $k$ to approximate
$\mathfrak{Q}_{t}(\cdot,\xi_{t j})$.  At the end of the backward pass of iteration $k$, the variant of MuDA with cut selection 
computes  
approximations $\mathcal{Q}_t^k$ of $\mathcal{Q}_t$ given by 
\eqref{approxQt} now with $\mathfrak{Q}_t^k( \cdot, \xi_{t j})$ given by 
\begin{equation} \label{qtk}
\displaystyle \mathfrak{Q}_{t}^k( x_{t-1}, \xi_{t j} ) =  \max_{\ell \in S_{t j}^k} \, \mathcal{C}_{t j}^\ell( x_{t-1}),
\end{equation}
where the set $S_{t j}^k$ is a subset of the set of indices of the cuts that
have the largest value for at least one of the trial points computed so far. 
More precisely, sets $S_{t j}^k$ are initialized taking $S_{t j}^0=\{0\}$.  
For $t \in \{2,\ldots,T\}$ and $k \geq 1$, sets $S_{t j}^k$ are computed as follows.
For $i=1,\ldots,k$, $t=2,\ldots,T$, $j=1,\ldots,M_t$, let
$I_{t j}^{i k}$ be the set of cuts for $\mathfrak{Q}_{t}(\cdot,\xi_{t j})$ computed at iteration $k$ or before, that have the largest value
at $x_{t-1}^i$:
\begin{equation} \label{indexlevel1}
I_{t j}^{i k} = \operatorname*{arg\,max}_{\ell=1,\ldots,k} \mathcal{C}_{t j}^{\ell}( x_{t-1}^i ),
\end{equation}
where the cut indices in $I_{t j}^{i k}$ are sorted in ascending order.
With a slight abuse of notation, we will denote the $\ell$-th smallest element in $I_{t j}^{i k}$ by $I_{t j}^{i k}(\ell)$.
For instance, if $I_{t j}^{i k}=\{2,30,50\}$ then $I_{t j}^{i k}(1)=2, I_{t j}^{i k}(2)=30$, and $I_{t j}^{i k}(3)=50$.
A cut selection strategy is given by a set of selectors $\mathcal{S}_{t j}(m)$, $t=2,\ldots,T$,$j=1,\ldots,M_t$, $m=1,2,\ldots,$ where 
$\mathcal{S}_{t j}(m)$ is a subset of the set 
of $m$ integers 
$\{1,2,\ldots,m\}$, giving the indices of the cuts to select in
$I_{t j}^{i k}$ through the relation
$$
S_{t j}^k = \bigcup_{i=1}^k \left\{I_{t j}^{i k}(\ell) : \ell \in \mathcal{S}_{t j}(| I_{t j}^{i k} |) \right\},
$$
where $| I_{t j}^{i k} |$ is the cardinality of set $I_{t j}^{i k}$.
We require the selectors to satisfy the following assumption:
\begin{itemize}
\item[(H3)] for $t=2,\ldots,T$, $j=1,\ldots,M_t$, for every $m \geq 1$, $\mathcal{S}_{t j}(m) \subseteq \mathcal{S}_{t j}(m+1)$.
\end{itemize}
Level 1 and Limited Memory Level 1 cut selection strategies described in Examples \ref{excs1} and \ref{excs2}
respectively correspond to the least and most economic selectors satisfying (H3):
\begin{example}[Level 1 and Territory Algorithm] \label{excs1} The strategy $\mathcal{S}_{t j}(m)=\{1,2,\ldots,m\}$ selects all cuts that have the highest value for at least
one trial point. In the context of SDDP, this strategy was called Level 1 in \cite{dpcuts0}  and
Territory Algorithm in \cite{pfeifferetalcuts}. For this strategy, we have $S_{T j}^k=\{1,\ldots,k\}$ for all $j$ and $k \geq 1$,
meaning that no cut selection is needed for the last stage $T$. This comes from the fact that for all  $k \geq 2$ and $1 \leq k_1 \leq k$,
cut $\mathcal{C}_{T j}^{k_1}$ is selected because it is one of the cuts with the highest value at $x_{T-1}^{k_1}$. Indeed,
for any $1 \leq k_2 \leq k$ with $k_2 \neq k_1$, 
since $\lambda_{T j}^{k_2}$ is feasible for problem \eqref{backwardT} with $x_{T-1}^{k}$ replaced by $x_{T-1}^{k_1}$, we
get
$$
\mathcal{C}_{T j}^{k_1}( x_{T-1}^{k_1} ) = \mathfrak{Q}_T( x_{T-1}^{k_1} , \xi_{T j} ) 
\geq \langle \lambda_{T j}^{k_2} , b_{T j} - B_{T j} x_{T-1}^{k_1} \rangle = \mathcal{C}_{T j}^{k_2}( x_{T-1}^{k_1} ).
$$
\end{example}
\begin{example}[Limited Memory Level 1]\label{excs2} The strategy that eliminates the largest amount of cuts, 
called Limited Memory Level 1 (LML 1 for short), consists in taking
a singleton for every set $\mathcal{S}_{t j}(m)$. For (H3) to be satisfied, this implies
$\mathcal{S}_{t j}(m)=\{1\}$. This choice corresponds to the Limited Memory Level 1 cut selection introduced in \cite{guiguesejor2017}
in the context of DDP. For that particular choice, at a given point, among the cuts that have the highest value
only the oldest (i.e., the cut that was first computed among the cuts that have the highest value at that point) is selected.
\end{example}

\begin{remark}
Observe that our selectors select cuts for functions $\mathfrak{Q}_t(\cdot,\xi_{t j})$
whereas Level 1 (resp. LML 1) was introduced to select cuts for
functions $\mathcal{Q}_t$ for SDDP (resp. DDP).
Therefore we could have used the terminologies Multicut Level 1 and Multicut Limited Memory Level 1
instead of Level 1 and Limited Memory Level 1. We did not do so to simplify and therefore
to know to which functions cut selection strategies apply, it suffices to add the name 
of the decomposition method; for instance MuDA with Level 1 cut selection (in which case cut selection applies
to functions  $\mathfrak{Q}_t(\cdot,\xi_{t j})$, as explained above) or SDDP with Limited Memory
Level 1 cut selection (in which case cut selection applies to functions $\mathcal{Q}_t$).
\end{remark}

The computation of $S_{t j}^k$, i.e., of the cut indices to select at iteration $k$, is performed in the backward pass
(immediately after computing cut $\mathcal{C}_{t j}^k$) using the pseudo-code given in the left and right  panels of Figure \ref{figurecut1} for
 the Level 1 and LML 1 cut selection strategies respectively.

In this pseudo-code, we use 
the notation $I_{t j}^i$ in place of $I_{t j}^{i k}$. We also store in variable
$m_{t j}^i$ the current value of the highest cut for $\mathfrak{Q}_{t}(\cdot, \xi_{t j})$
at $x_{t-1}^{i}$. At the end of the first iteration, we initialize
$m_{t j}^1 = \mathcal{C}_{t j}^1( x_{t-1}^{1})$. After cut $\mathcal{C}_{t j}^k$ is computed at 
iteration $k \geq 2$, these variables are updated using the relations
$$   
\left\{
\begin{array}{lll}
m_{t j}^i &  \leftarrow & \max (m_{t j}^i , \mathcal{C}_{t j}^k ( x_{t-1}^{i}) ),\;i=1,\ldots,k-1,\\
m_{t j}^k &  \leftarrow & \max (\mathcal{C}_{t j}^{\ell }( x_{t-1}^{k}), \ell=1,\ldots,k ).
\end{array}
\right. 
$$
We also use an array  of Boolean called {\tt{Selected}} using the information given by variables $I_{t j}^i$
whose $\ell$-th entry is {\tt{True}} if cut $\ell$ is selected for $\mathfrak{Q}_{t}(\cdot, \xi_{t j})$ and {\tt{False}} otherwise. 
This allows us  to avoid copies of cut indices that may appear in $I_{t {j}}^{i_1 k}$ and $I_{t {j}}^{i_2 k}$ with $i_1 \neq i_2$.
\begin{figure}
\begin{tabular}{|c|c|}
 \hline 
Level 1 & Limited Memory Level 1 \\
\hline
\begin{tabular}{l}
$I_{t j}^{k}=\{k\}$, $m_{t  j}^k =\mathcal{C}_{t j}^k ( x_{t-1}^k )$.\\
{\textbf{For}} $\ell=1,\ldots,k-1$,\\
\hspace*{0.3cm}{\textbf{If }}$\mathcal{C}_{t j}^k( x_{t-1}^{\ell} ) > m_{t j}^{\ell}$ \\
\hspace*{0.6cm}$I_{t j}^{\ell}=\{k\},\; m_{t j}^{\ell}=\mathcal{C}_{t j}^k( x_{t-1}^{\ell} )$\\
\hspace*{0.3cm}{\textbf{Else if}} $\mathcal{C}_{t j}^k( x_{t-1}^{\ell} ) = m_{t j}^{\ell}$ \\
\hspace*{0.6cm}$I_{t j}^{\ell}=I_{t j}^{\ell} \cup \{k\}$\\
\hspace*{0.3cm}{\textbf{End If}}\\
\hspace*{0.3cm}{\textbf{If }}$\mathcal{C}_{t j}^{\ell}( x_{t-1}^{k} ) > m_{t j}^k$ \\
\hspace*{0.6cm}$I_{t j}^k =\{\ell\},\; m_{t j}^k = \mathcal{C}_{t j}^{\ell}( x_{t-1}^{k} )$\\
\hspace*{0.3cm}{\textbf{Else if}} $\mathcal{C}_{t j}^{\ell}( x_{t-1}^{k} ) = m_{t j}^{k}$ \\
\hspace*{0.6cm}$I_{t j}^{k}=I_{t j}^{k} \cup \{\ell\}$\\
\hspace*{0.3cm}{\textbf{End If}}\\
{\textbf{End For}}\\
{\textbf{For}} $\ell=1,\ldots,k$,\\
\hspace*{0.3cm}{\tt{Selected}}[$\ell$]={\tt{False}}\\
{\textbf{End For}}\\
{\textbf{For}} $\ell=1,\ldots,k$\\
\hspace*{0.3cm}{\textbf{For}} $m=1,\ldots,|I_{t j}^{\ell}|$\\
\hspace*{0.6cm}{\tt{Selected}}[$I_{t j}^{\ell}[m]$]={\tt{True}}\\
\hspace*{0.3cm}{\textbf{End For}}\\
{\textbf{End For}}\\
$S_{t j}^k=\emptyset$\\
{\textbf{For}} $\ell=1,\ldots,k$\\
\hspace*{0.3cm}{\textbf{If}} {\tt{Selected}}[$\ell$]={\tt{True}} \\
\hspace*{0.6cm}$S_{t j}^k=S_{t j}^k \cup \{\ell\}$\\
\hspace*{0.3cm}{\textbf{End If}}\\
{\textbf{End For}}\\
\end{tabular}
&
\begin{tabular}{l}
\vspace*{-0.45cm}\\
$I_{t j}^{k}=\{1\}$, $m_{t  j}^k =\mathcal{C}_{t j}^1 ( x_{t-1}^k )$.\\
{\textbf{For}} $\ell=1,\ldots,k-1$,\\
\hspace*{0.3cm}{\textbf{If }}$\mathcal{C}_{t j}^k( x_{t-1}^{\ell} ) > m_{t j}^{\ell}$ \\
\hspace*{0.6cm}$I_{t j}^{\ell}=\{k\},\; m_{t j}^{\ell}=\mathcal{C}_{t j}^k( x_{t-1}^{\ell} )$\\
\hspace*{0.3cm}{\textbf{End If}}\\
\vspace*{0.4cm}\\
\hspace*{0.3cm}{\textbf{If }}$\mathcal{C}_{t j}^{\ell + 1}( x_{t-1}^{k} ) > m_{t j}^k$ \\
\hspace*{0.6cm}$I_{t j}^k =\{\ell+1\},\; m_{t j}^k = \mathcal{C}_{t j}^{\ell + 1}( x_{t-1}^{k} )$\\
\hspace*{0.3cm}{\textbf{End If}}\\
\vspace*{0.4cm}\\
{\textbf{End For}}\\
{\textbf{For}} $\ell=1,\ldots,k$,\\
\hspace*{0.3cm}{\tt{Selected}}[$\ell$]={\tt{False}}\\
{\textbf{End For}}\\
{\textbf{For}} $\ell=1,\ldots,k$\\
\hspace*{0.3cm}{\textbf{For}} $m=1,\ldots,|I_{t j}^{\ell}|$\\
\hspace*{0.6cm}{\tt{Selected}}[$I_{t j}^{\ell}[m]$]={\tt{True}}\\
\hspace*{0.3cm}{\textbf{End For}}\\
{\textbf{End For}}\\
$S_{t j}^k=\emptyset$\\
{\textbf{For}} $\ell=1,\ldots,k$\\
\hspace*{0.3cm}{\textbf{If}} {\tt{Selected}}[$\ell$]={\tt{True}} \\
\hspace*{0.6cm}$S_{t j}^k=S_{t j}^k \cup \{\ell\}$\\
\hspace*{0.3cm}{\textbf{End If}}\\
{\textbf{End For}}\\
\end{tabular}
\\
\hline
\end{tabular}

\caption{Pseudo-codes for the computation of set $S_{t j}^k$ for fixed $t\in \{2,\ldots,T\}, k \geq 2, j=1,\ldots,M_t$, and two cut selection strategies.}
\label{figurecut1}
\end{figure}

\section{Convergence analysis}\label{sec:convanalysis}

In this section, we prove that CuSMuDA converges in a finite number of iterations.
We will make the following assumption:
\begin{itemize}
\item[(H4)] The samples in the forward passes are independent: $(\tilde \xi_2^k, \ldots, \tilde \xi_T^k)$ is a realization of
$\xi^k=(\xi_2^k, \ldots, \xi_T^k) \sim (\xi_2, \ldots, \xi_T)$ 
and $\xi^1, \xi^2,\ldots,$ are independent.
\end{itemize}
The convergence proof is based on the following lemma:
\begin{lemma} \label{lemmafinitecuts}
Assume that all subproblems in the forward and backward passes
of CuSMuDA
are solved using an algorithm
that necessarily outputs an extreme point of the feasible set (for instance the simplex algorithm).
Let assumptions (H1), (H2), (H3), and (H4) hold.
Then almost surely, there exists $k_0 \geq 1$ such that 
for every $k \geq k_0$, $t=2,\ldots,T$, $j=1,\ldots,M_t$, we have
\begin{equation} \label{ineqk0}
\mathfrak{Q}_{t}^k (\cdot, \xi_{t j})=\mathfrak{Q}_{t}^{k_0}(\cdot, \xi_{t j})
\mbox{ and }\mathcal{Q}_{t}^k =\mathcal{Q}_{t}^{k_0}.
\end{equation}
\end{lemma}
\begin{proof} 
Let $\Omega_1$ be the event on the sample space $\Omega$ of sequences of forward scenarios such that 
every scenario is sampled an infinite number of times.
By Assumption (H4), this event $\Omega_1$ has probability one.

Consider a realization $\omega \in \Omega$ of CuSMuDA in $\Omega_1$ corresponding to realizations $(\tilde \xi^k_{1:T})_k$ of  $(\xi^k_{1:T})_k$ in the forward pass. To simplify, we will drop $\omega$ in the notation. For instance,
we will simply write $x_t^k, \mathcal{Q}_t^k$  for realizations  $x_t^k( \omega)$ and $\mathcal{Q}_t^k(\cdot)(\omega)$
of $x_t^k, \mathcal{Q}_t^k$ given realization $\omega \in \Omega$ of CuSMuDA.

\if{
\paragraph{{\textbf{1$^0$}}} We show by induction that the set of possible trial points $\{x_{t}^k : k \geq 1 \}$ is finite for every $t=1,\ldots,T-1$.
\paragraph{{\textbf{1$^0.$a.}}} $x_1^k$ is an extreme point of the polyhedron $X_1(x_0, \xi_1)$ which
has a finite number of extreme points since $A_1$ has a finite number of rows and columns.
Therefore $x_1^k$ can only take a finite number of values.
\paragraph{{\textbf{1$^0.$b.}}} Assume now that $x_{t-1}^k$ can only take a finite number of values for some $t\in \{2,\ldots,T\}$.
Recall that $x_t^k$ is an extreme point of one of the sets $X_t(x_{t-1}^k, \xi_{t j})$ where $j \in \{1,\ldots,M_t\}$.
Since each matrix $A_{t j}$ has a finite number of
rows and columns, each set  $X_t(x_{t-1}^k, \xi_{t j})$ has a finite number of extreme points.
Using the induction hypothesis and $M_t<+\infty$, there is a finite number of these sets and therefore
$x_t^k$ can only take a finite number of  values. This shows {\textbf{1$^0$.}}
}\fi
\par  We show by induction on $t$ that the number of different cuts computed by the algorithm is finite and that
after some iteration $k_t$ the same cuts are selected for functions $\mathfrak{Q}_{t}(\cdot, \xi_{t j})$.
Our induction hypothesis $\mathcal{H}(t)$ for $t\in \{2,\ldots,T\}$ is
that the sets $\{(\theta_{t j}^k , \beta_{t j}^k) : k \in \mathbb{N}\}, j=1,\ldots,M_t$, are finite and there exists some
finite $k_t$ such that for every $k>k_t$ we have
\begin{equation}\label{inductionT}
\{(\theta_{t j}^{\ell} , \beta_{t j}^{\ell}) : \ell \in S_{t j}^k \}=
\{(\theta_{t j}^{\ell} , \beta_{t j}^{\ell}) : \ell \in S_{t j}^{k_t} \}
\mbox{ and }x_{t-1}^k = x_{t-1}^{k_t},
\end{equation}
for every $j=1,\ldots,M_t$.
We will denote by $\mathcal{I}_{t j}^{i k}$ the set $\left\{I_{t j}^{i k}(\ell) : \ell \in \mathcal{S}_{t j}(| I_{t j}^{i k} |) \right\}$.
We first show, in items {{\textbf{a}}} and {{\textbf{b}}}  below that $\mathcal{H}(T)$ holds.
\par {{\textbf{a.}}} Observe that $\lambda_{T j}^k$ defined in the backward pass of CuSMuDA is an extreme point of the polyhedron
$\{\lambda : A_{T j}^\top \lambda \leq c_{T j}\}$. 
This polyhedron is a finite intersection of closed half spaces in finite dimension (since
$A_{T j}$ has a finite number of rows and columns) and therefore has a finite number of extreme points. It follows that $\lambda_{T j}^k$ can only take a finite
number of values, same as $(\theta_{T j}^k, \beta_{T j}^k)=(\langle \lambda_{T j}^k, b_{T j} \rangle , -B_{T j}^\top \lambda_{T j}^k)$, and
there exists ${\bar k}_{T}$ such that for every $k>{\bar k}_T$ and every $j$, each cut $\mathcal{C}_{T j}^k$ is a copy of a cut $\mathcal{C}_{T j}^{k'}$ with $1 \leq k' \leq {\bar k}_T$ (no new
cut is computed for functions $\mathfrak{Q}_{T}(\cdot, \xi_{T j})$ for $k>{\bar k}_T$).

Now recall that $x_{T-1}^k$ computed in the forward pass is a solution of
\eqref{pbforwardpass} with $t=T-1$ and this optimization problem
can be written as a linear problem 
adding variables $f_1,f_2,\ldots,f_{M_T}$
replacing in the objective $\mathcal{Q}_{T}^{k-1}(x_{T-1})$
by $\sum_{\ell=1}^{M_T} p_{T \ell} f_{\ell}$ and adding the linear
constraints 
$f_{\ell} \geq \theta_{T \ell}^i + \langle \beta_{T \ell}^i , x_{T-1} \rangle$,
$i=1,\ldots,k-1,$ $\ell=1,\ldots,M_T$.
On top of that, for iterations $k>{\bar k}_T$, since functions 
$\mathfrak{Q}_T^k(\cdot,\xi_{T j})$ are made of a collection 
of cuts taken from
the finite and fixed
set of cuts $\mathcal{C}_{T j}^{\ell}, \ell \leq \bar k_T$, the set of 
possible functions 
$(\mathfrak{Q}_T^k(\cdot,\xi_{T j}))_{k \geq 1}$ and therefore of possible functions 
$(\mathcal{Q}_T^k)_{k \geq 1}$ is finite.
It follows that there is a finite set of possible polyhedrons for the feasible
set of \eqref{pbforwardpass} (with $t=T-1$) rewritten as a linear program as we have just
explained, adding variables $f_1,f_2,\ldots,f_{M_T}$.
Since these polyhedrons have a finite number of extreme points (recall that
there is a finite number of different linear constraints), there is only
a finite number of possible trial points $(x_{T-1}^k)_{k \geq 1}$.
Therefore we can assume without loss of generality that $\bar k_T$ is such that 
for iterations $k>\bar k_T$, all trial point $x_{T-1}^k$ is also a copy of a trial point
$x_{T-1}^{k'}$ with $k' \leq \bar k_T$.

\par {\textbf{b}}. We show that for every $i \geq 1$ and $j=1,\ldots,M_T$,
there exists $1 \leq i' \leq {\bar k}_T$ 
and  $k_{T j}(i') \geq  {\bar k}_T$
such that for every $k \geq  \max(i, k_{T j}(i'))$ we have
\begin{equation}\label{inductioncruc}
\left\{ (\theta_{T j}^{\ell}, \beta_{T j}^{\ell}) :  \ell \in \mathcal{I}_{T j}^{i k} \right\}=
\left\{ (\theta_{T j}^{\ell}, \beta_{T j}^{\ell}) :  \ell \in \mathcal{I}_{T j}^{i' k_{T j}(i')} \right\},
\end{equation}
which will show $\mathcal{H}(T)$ with  $k_T = \max_{1 \leq  j \leq M_T, 1 \leq i' \leq {\bar k}_T} k_{T j}( i')$.
Let us show that \eqref{inductioncruc} indeed holds.
Let us take $i \geq 1$ and $j \in \{1,\ldots,M_T \}$.
If $1 \leq i \leq {\bar k}_{T}$ define $i'=i$. Otherwise
due to {\textbf{a.}} we can find $1 \leq i' \leq {\bar k}_T$ such that 
$x_{T-1}^i = x_{T-1}^{i'}$ which implies $I_{T j}^{i k}= I_{T j}^{i' k}$ for every $k \geq i$.
Now consider the sequence of sets $(I_{T j}^{i' k})_{k \geq  {\bar k}_T}$.
Due to the definition of $\bar k_T$, the sequence  $(|I_{T j}^{i' k}|)_{k > {\bar k}_T}$ is nondecreasing and therefore
two cases can happen:
\begin{itemize}
\item[(A)] there exists $k_{T j}(i')$ such that for $k \geq k_{T j}(i')$ we have $I_{T j}^{i' k}=I_{T j}^{i' k_{T j}(i')}$ (the cuts
computed after iteration $k_{T j}(i')$ are not active at $x_{T-1}^{i'}$). In this case 
$\mathcal{I}_{T j}^{i' k} = \mathcal{I}_{T j}^{i' k_{T j}(i')}$
for $k \geq k_{T j}(i')$, $\mathcal{I}_{T j}^{i k} = \mathcal{I}_{T j}^{i' k} = \mathcal{I}_{T j}^{i' k_{T j}(i')}$
for $k \geq \max( k_{T j}(i'), i)$ and 
\eqref{inductioncruc} holds.
\item[(B)] The sequence $(|I_{T j}^{i' k}|)_{k \geq  {\bar k}_T}$ is unbounded.
Due to Assumption (H3), the sequence $(|\mathcal{S}_{T j}(m)|)_m$ is nondecreasing.
If there exists $k_{T j}>\bar k_T$ such that  $\mathcal{S}_{T j}(k)=\mathcal{S}_{T j}(k_{T j})$ for $k \geq k_{T j}$
then if $k_{T j}( i' )$ is the smallest $k$ such that 
$|I_{T j}^{i' k}| \geq k_{T j}$ then for every $k \geq k_{T j}( i' )$ we have 
$\mathcal{I}_{T j}^{i' k} = \mathcal{I}_{T j}^{i' k_{T j}(i')}$ and
for $k \geq \max( k_{T j}(i'), i)$ we deduce 
$\mathcal{I}_{T j}^{i k} = \mathcal{I}_{T j}^{i' k_{T j}(i')}$
and \eqref{inductioncruc} holds.
Otherwise the sequence $(|\mathcal{S}_{T j}(m)|)_m$ is unbounded and an infinite number of cut indices
are selected from the sets $(I_{T j}^{i' k})_{k \geq i'}$ to make up sets $(\mathcal{I}_{T j}^{i' k})_{k \geq i'}$. 
However, since there is a finite number of different cuts, there is only a finite number of iterations where a new
cut can be selected from $I_{T j}^{i' k}$ and therefore there exists $k_{T j}(i')$ such that \eqref{inductioncruc} holds
for $k \geq \max( k_{T j}(i'), i)$.
\end{itemize}
\paragraph{{\textbf{c.}}}
Now assume that $\mathcal{H}(t+1)$ holds for some $t \in \{2,\ldots,T-1\}$. We want to show $\mathcal{H}(t)$.
Consider the set $\mathcal{D}_{t j k}$ of points of form
\begin{equation}\label{relationthetabeta}
(\langle  \lambda ,  b_{t j} \rangle +  \sum_{\ell=1}^{M_{t+1}} \sum_{i \in S_{t+1 \ell}^k} \mu_{i \ell} \theta_{t+1 \ell}^i , - B_{t j}^\top \lambda )
\end{equation}
 where $(\lambda, \mu )$ is an extreme point of 
the set $\mathcal{P}_{t j k}$ of points $(\lambda, \mu)$
satisfying 
\begin{equation}\label{extpointpoly}
\begin{array}{l}
\mu \geq 0, p_{t+1 \ell}=\sum_{i \in S_{t+1 \ell}^k} \mu_{i \ell},\;\ell=1,\ldots,M_{t+1}, \;A_{t j}^\top \lambda + \sum_{\ell =1}^{M_{t+1}} \sum_{i \in S_{t+1 \ell}^k} \mu_{i \ell} \beta_{t+1 \ell}^i \leq c_{t j}.
\end{array}
\end{equation}
We claim that for every $k>k_{t+1}$, every point from $\mathcal{D}_{t j k}$
can be written as a point from $\mathcal{D}_{t j k_{t+1}}$, i.e., a point
of form \eqref{relationthetabeta}  with $k$ replaced by $k_{t+1}$ and $(\lambda, \mu )$ an extreme point of 
the set $\mathcal{P}_{t j k_{t+1}}$. Indeed, take a point from $\mathcal{D}_{t j k}$, i.e., a point
of form \eqref{relationthetabeta}
with $(\lambda, \mu)$ an extreme point of $\mathcal{P}_{t j k}$ and $k>k_{t+1}$. 
It can be written as a point from $\mathcal{D}_{t j k_{t+1}}$, i.e., a point of form 
\eqref{relationthetabeta} with $k$ replaced by $k_{t+1}$ and $(\lambda, \hat \mu)$ in the place
of $(\lambda, \mu)$
where $(\lambda, \hat \mu)$ is an extreme point of $\mathcal{P}_{t j k_{t+1}}$  obtained 
replacing the basic columns $\beta_{t+1 \ell}^i$ with $i \in S_{t+1 \ell}^k, i \notin S_{t+1 \ell}^{k_{t+1}}$
associated with $\mu$
by columns $\beta_{t+1 \ell}^{i'}$ with $i' \in S_{t+1 \ell}^{k_{t+1}}$
such that $(\theta_{t+1 \ell}^{i}, \beta_{t+1 \ell}^{i})=(\theta_{t+1 \ell}^{i'}, \beta_{t+1 \ell}^{i'})$ (this is possible due to $\mathcal{H}(t+1)$).

Since $\mathcal{P}_{t j k_{t+1}}$ has a finite number of extreme points,
the set $\mathcal{D}_{t j k}$ has a finite cardinality and recalling that
for CuSMuDA $(\theta_{t j}^k, \beta_{t j}^k) \in \mathcal{D}_{t j k}$, the cut coefficients $(\theta_{t j}^k, \beta_{t j}^k)$
can only take a finite number of values. This shows the first part of $\mathcal{H}(t)$. 
Therefore, there exists ${\bar k}_{t}$ such that for every $k>{\bar k}_t$ and every $j$, each cut $\mathcal{C}_{t j}^k$ is a copy of a cut $\mathcal{C}_{t j}^{k'}$ with $1 \leq k' \leq {\bar k}_t$ (no new
cut is computed for functions $\mathfrak{Q}_{t}(\cdot, \xi_{t j})$ for $k>{\bar k}_t$).
As for the induction step $t=T$, this clearly implies that
after some iteration, no new trial points are computed and therefore we can assume without loss
of generality that $\bar k_t$ is such that 
for $k>\bar k_T$ trial point $x_{t-1}^k$ is a copy of some $x_{t-1}^{\ell}$ with $1 \leq \ell \leq {\bar k}_t$.

Finally, we can show \eqref{inductionT}
proceeding as in {\textbf{b.}}, replacing $T$ by $t$. This achieves the proof of $\mathcal{H}(t)$.

Gathering our observations, we have shown that \eqref{ineqk0} holds with $k_0 = \max_{t=2,\ldots,T} k_t$.
\end{proof}
\begin{remark}
For Level 1 and LML 1 cut selection strategies corresponding to selectors
$\mathcal{S}_{t j}$  satisfying respectively $\mathcal{S}_{t j}(m)=\{1,\ldots,m\}$
and $\mathcal{S}_{t j}(m)=\{1\}$, integers $k_0, {\bar k}_t$ defined in
Lemma \ref{lemmafinitecuts} and its proof satisfy $k_0 \leq \max_{t=2,\ldots,T} {\bar k}_t$.
For other selectors $\mathcal{S}_{t j}$ this relation is not necessarily satisfied (see {\textbf{b.}}-(B) of the proof
of the lemma).
\end{remark}

\begin{theorem}\label{convproof} Let Assumptions (H1), (H2), (H3), and (H4) hold.
Assume that all subproblems in the forward and backward passes
of CuSMuDA
are solved using an algorithm
that necessarily outputs an extreme point of the feasible set (for instance the simplex algorithm).
Then Algorithm CuSMuDA converges with probability one in a finite number of iterations to a policy which is an optimal solution
of \eqref{firststodp}-\eqref{secondstodp}. 
\end{theorem}
\begin{proof} Let $\Omega_1$ be defined as in the proof of Lemma \ref{lemmafinitecuts}  and  
let $\Omega_2$ be the event such that $k_0$ defined in Lemma \ref{lemmafinitecuts} is finite.
Note that $\Omega_1 \cap \Omega_2$ has probability $1$.
Consider a realization of CuSMuDA in $\Omega_1 \cap \Omega_2$ corresponding to realizations
$(\tilde \xi^k_{1:T})_k$ of  $(\xi^k_{1:T})_k$ in the forward pass 
and let $(x_1^*, x_2^*( \cdot )$, $\ldots , x_T^*( \cdot ) )$
be the policy obtained from iteration $k_0$ on which uses recourse functions $\mathcal{Q}_{t+1}^{k_0}$ instead of $\mathcal{Q}_{t+1}$.
Recall that policy $(x_1^*, x_2^*( \cdot )$, $\ldots , x_T^*( \cdot ) )$ is optimal if for every realization
${\tilde \xi}_{1:T} := (\xi_1, {\tilde \xi}_2, \ldots,{\tilde \xi}_T)$ of $\xi_{1:T}:=(\xi_1, \xi_2,\ldots,\xi_T)$, we have that $x_t^* (\xi_1, {\tilde \xi}_2, \ldots,{\tilde \xi}_t )$
solves
\begin{equation}\label{optline}
\mathfrak{Q}_t ( x_{t-1}^{*}( {\tilde \xi}_{1:t-1} ), {\tilde \xi}_t   ) =
\displaystyle \inf_{x_t} \{ \langle  {\tilde c}_t ,  x_t \rangle + \mathcal{Q}_{t+1}( x_t) \;:\; x_t \in X_t( x_{t-1}^* ({\tilde \xi}_{1:t-1})  , {\tilde \xi}_t )    \}  
\end{equation}
for every $t=1,\ldots,T$, with the convention that $x_{0}^* = x_0$. 
We prove for $t=1,\ldots,T$, 
$$
\begin{array}{l}
{\overline{\mathcal{H}}}(t): \mbox{ for every }k \geq k_0 \mbox{ and for every sample }{\tilde \xi}_{1:t}=(\xi_1, {\tilde \xi}_2, \ldots,{\tilde \xi}_t) \mbox{ of }(\xi_1,\xi_2,\ldots,\xi_t), \mbox{ we have}\\
\hspace*{1.1cm}{\underline{\mathfrak{Q}}}_t^{k} ( x_{t-1}^{*}( {\tilde \xi}_{1:t-1} ) , {\tilde \xi}_{t}  )= \mathfrak{Q}_t ( x_{t-1}^{*}( {\tilde \xi}_{1:t-1} ) , {\tilde \xi}_t  ).
\end{array}
$$
We show ${\overline{\mathcal{H}}}(1),\ldots,{\overline{\mathcal{H}}}(T)$ by induction.
${\overline{\mathcal{H}}}(T)$ holds since 
${\underline{\mathfrak{Q}}}_T^{k} = \mathfrak{Q}_T$
for every $k$. 
Now assume that ${\overline{\mathcal{H}}}(t+1)$ holds for some $t \in \{1,\ldots,T-1\}$.
We want to show ${\overline{\mathcal{H}}}(t)$.
Take an arbitrary $k \geq k_0$ and a  sample  ${\tilde \xi}_{1:t-1}=(\xi_1,{\tilde \xi}_2, \ldots,{\tilde \xi}_{t-1})$ of
$(\xi_1,\xi_2,\ldots,\xi_{t-1})$. We have for every $j=1,\ldots,M_t$, that
\begin{equation}\label{eqQtapprox}
\begin{array}{lll}
{\underline{\mathfrak{Q}}}_t^{k} ( x_{t-1}^{*}( {\tilde \xi}_{1:t-1} ) , \xi_{t j}  )&=&
\left\{
\begin{array}{l}
\inf \;\langle c_{t j} , x_t \rangle  + \mathcal{Q}_{t+1}^k ( x_t )\\
x_t \in X_t( x_{t-1}^{*}( {\tilde \xi}_{1:t-1} ), \xi_{t j})
\end{array}
\right.\\
&=& \langle c_{t j} ,  x_{t}^{*}( {\tilde \xi}_{1:t-1}, \xi_{t j} ) \rangle + \mathcal{Q}_{t+1}^{k} ( x_{t}^{*}( {\tilde \xi}_{1:t-1}, \xi_{t j} ) ).
\end{array}
\end{equation}
Now we check that for $j=1,\ldots,M_t$, we have
\begin{equation}\label{indcstep}
\mathcal{Q}_{t+1}^{k} ( x_{t}^{*}( {\tilde \xi}_{1:t-1}, \xi_{t j} ) ) = \mathcal{Q}_{t+1}  ( x_{t}^{*}( {\tilde \xi}_{1:t-1}, \xi_{t j} ) ).
\end{equation}
Indeed, if this relation did not hold, since 
$\mathcal{Q}_{t+1}^{k}  \leq \mathcal{Q}_{t+1} $, we would have 
$$
\mathcal{Q}_{t+1}^{k_0} ( x_{t}^{*}( {\tilde \xi}_{1:t-1}, \xi_{t j} ) ) =\mathcal{Q}_{t+1}^{k} ( x_{t}^{*}( {\tilde \xi}_{1:t-1}, \xi_{t j} ) ) < \mathcal{Q}_{t+1}  ( x_{t}^{*}( {\tilde \xi}_{1:t-1}, \xi_{t j} ) ).
$$
From the definitions of $\mathcal{Q}_{t+1}^k, \mathcal{Q}_{t+1}$, there exists $m \in \{1,\ldots,M_{t+1}\}$ such that
$$
\mathfrak{Q}_{t+1}^{k_0}( x_{t}^{*}( {\tilde \xi}_{1:t-1}, \xi_{t j} )  , \xi_{t+1 m})<\mathfrak{Q}_{t+1}( x_{t}^{*}( {\tilde \xi}_{1:t-1}, \xi_{t j} )  , \xi_{t+1 m}).
$$
Since the realization of CuSMuDA is in $\Omega_1$, there exists an infinite set of iterations such that the sampled scenario for stages 
$1,\ldots,t$, is $(\tilde \xi_{1:t-1}, \xi_{t j})$. Let $\ell$ be one of these iterations strictly greater than $k_0$.
Using ${\overline{\mathcal{H}}}(t+1)$, we have that 
$$
{\underline{\mathfrak{Q}}}_{t+1}^{\ell}( x_{t}^{*}( {\tilde \xi}_{1:t-1}, \xi_{t j} )  , \xi_{t+1 m}) =\mathfrak{Q}_{t+1}( x_{t}^{*}( {\tilde \xi}_{1:t-1}, \xi_{t j} )  , \xi_{t+1 m})
$$
which yields
$$
\mathfrak{Q}_{t+1}^{k_0}( x_{t}^{*}( {\tilde \xi}_{1:t-1}, \xi_{t j} )  , \xi_{t+1 m})<{\underline{\mathfrak{Q}}}_{t+1}^{\ell}( x_{t}^{*}( {\tilde \xi}_{1:t-1}, \xi_{t j} )  , \xi_{t+1 m})
$$
and at iteration $\ell>k_0$ we would construct a cut for $\mathfrak{Q}_{t+1}(\cdot, \xi_{t+1 m})$ at $x_{t}^{*}( {\tilde \xi}_{1:t-1}, \xi_{t j} )=x_t^{\ell}$
with value 
$${\underline{\mathfrak{Q}}}_{t+1}^{\ell}( x_{t}^{*}( {\tilde \xi}_{1:t-1}, \xi_{t j} )  , \xi_{t+1 m})$$ strictly larger than the value
at this point of all cuts computed up to iteration $k_0$.
Due to Lemma \ref{lemmafinitecuts}, this is not possible. Therefore, \eqref{indcstep} holds, which, plugged into \eqref{eqQtapprox}, gives
$$
{\underline{\mathfrak{Q}}}_t^{k} ( x_{t-1}^{*}( {\tilde \xi}_{1:t-1} ) , \xi_{t j}  )=
\langle  c_{t j}  ,  x_{t}^{*}( {\tilde \xi}_{1:t-1}, \xi_{t j} ) \rangle + \mathcal{Q}_{t+1} ( x_{t}^{*}( {\tilde \xi}_{1:t-1}, \xi_{t j} ) )
\geq \mathfrak{Q}_t ( x_{t-1}^{*}( {\tilde \xi}_{1:t-1} ) , \xi_{t j}  )
$$
(recall that $x_{t}^{*}( {\tilde \xi}_{1:t-1}, \xi_{t j} ) \in X_t( x_{t-1}^{*}( {\tilde \xi}_{1:t-1}),   \xi_{t j})$).
Since ${\underline{\mathfrak{Q}}}_t^{k} \leq \mathfrak{Q}_t$, we have shown
$ {\underline{\mathfrak{Q}}}_t^{k} ( x_{t-1}^{*}( {\tilde \xi}_{1:t-1} ) , \xi_{t j}  )= \mathfrak{Q}_t ( x_{t-1}^{*}( {\tilde \xi}_{1:t-1} ) , \xi_{t j}  )$
for every $j=1,\ldots,M_t$, which is ${\overline{\mathcal{H}}}(t)$. 
Therefore, we have proved that for $t=1,\ldots,T,$ for every realization $(\tilde \xi_{1:t-1}, \xi_{t j})$ of $\xi_{1:t}$, $x_{t}^{*}( {\tilde \xi}_{1:t-1}, \xi_{t j} )$
satisfies 
$\langle c_{t j} ,   x_{t}^{*}( {\tilde \xi}_{1:t-1}, \xi_{t j} ) \rangle  + \mathcal{Q}_{t+1}( x_{t}^{*}( {\tilde \xi}_{1:t-1}, \xi_{t j} ) ) = \mathfrak{Q}_t ( x_{t-1}^{*}( {\tilde \xi}_{1:t-1} ) , \xi_{t j}  )$,
meaning that for every $j=1,\ldots,M_t$,
$x_{t}^{*}( {\tilde \xi}_{1:t-1}, \xi_{t j} )$
is an optimal solution of \eqref{optline} written with $\tilde \xi_{1:t} = (\tilde \xi_{1:t-1}, \xi_{t j})$ and completes the proof.
\end{proof}
\begin{remark} The convergence proof above also shows the almost sure convergence in a finite
number of iterations for MuDA combined with
cut selection strategies that would always select more cuts
than any cut selection strategy satisfying Assumption (H3).
It shows in particular the convergence of Level $H$ cut selection from \cite{dpcuts0} which
keeps the $H$ cuts having the largest values at each trial point.
\end{remark}
\begin{remark} 
The class of cut selection strategies described in Section \ref{sec:cutselectionalgo} can be straightforwardly combined with SDDP.
The convergence proof of the corresponding variant of SDDP applied to DP equations 
\eqref{firststodp}, \eqref{secondstodp}
can be easily obtained adapting the proofs of Lemma \ref{lemmafinitecuts} and Theorem \ref{convproof}. 
\end{remark}

\section{Application to portfolio selection and inventory management}\label{sec:numsim}

\subsection{Portfolio selection} \label{portfolio}

\subsubsection{Model} \label{portfoliomodel}

We consider a portfolio selection problem with direct transaction costs over a discretized horizon
of $T$ stages. The direct buying and selling  transaction costs are
proportional to the amount of the transaction 
(\cite{bentalmargnem00}, \cite{BEHL}).\footnote{This portfolio problem was solved using SDDP and SREDA (Stochastic REgularized Decomposition
Algorithm) in \cite{guilejtekregsddp}.} 
Let $x_t( i )$ be the dollar value of asset $i=1,\ldots,n+1$ at the end of stage $t=1,\ldots,T$,
where asset $n+1$ is cash; $\xi_t(i)$ is the return of asset $i$ at $t$; 
$y_t(i)$ is the amount of asset $i$ sold at the end of $t$; 
$z_t(i)$ is the amount of asset $i$ bought at the end of $t$, 
$\eta_t(i) > 0$ and $\nu_t(i) > 0$ are respectively the proportional selling and buying transaction costs at $t$.
Each component $x_0(i),i=1,\ldots,n+1$, of $x_0$ is known.
The budget available at the beginning of the investment period is
$\sum_{i=1}^{n+1} \xi_{1}(i) x_{0}(i)$ and
$u(i)$ represents the maximal proportion of money that can be invested in asset $i$. 

For $t=1,\ldots,T$, given a portfolio $x_{t-1}=(x_{t-1}(1),\ldots,x_{t-1}(n), x_{t-1}( n+1) )$ and $\xi_t$, we define the set $X_t(x_{t-1}, \xi_t)$
as the set of $(x_t, y_t, z_t) \in \mathbb{R}^{n+1} \small{\times} \mathbb{R}^{n} \small{\times} \mathbb{R}^{n}$ satisfying
\begin{equation}
x_t( n+1 ) = \xi_{t}( n+1 ) x_{t-1}( n+1 )   +\sum\limits_{i=1}^{n} \Big((1-\eta_t( i) )y_t( i )-  (1+\nu_t ( i ))z_t( i )\Big), \label{p1_3}
\end{equation}
and for $i=1,\ldots,n$,

\begin{subequations}\label{MI-CONS0}
\begin{align}
x_{t}(i)&= \xi_{t}(i) x_{t-1}(i)-y_t(i) +z_t(i),  \label{p1_2}\\
x_t(i) &\leq u(i)  \sum\limits_{j=1}^{n+1} \xi_{t}(j) x_{t-1}(j),  \label{p1_5}\\
x_t(i) &\ge 0, y_t(i) \geq 0, z_t(i) \geq 0.    \label{p1_6}
\end{align}
\end{subequations}
%\begin{eqnarray}
%\end{eqnarray}
% We will have here to define $u_i$ by taking into account the low number of securities considered.
Constraints \eqref{p1_3} are the cash flow balance constraints and define how much cash is available at each stage.
Constraints \eqref{p1_2} define the amount of security $i$ held at each stage $t$ and take into account the proportional transaction costs. 
Constraints \eqref{p1_5} prevent the position in security $i$ at time $t$ from exceeding a proportion $u(i)$.
Constraints \eqref{p1_6} prevent short-selling and enforce the non-negativity of the amounts bought and sold. 

With this notation, the following dynamic programming equations of a risk-neutral portfolio model can be written:
for $t=T$, setting $\mathcal{Q}_{T+1}( x_T ) = \mathbb{E}[ \sum\limits_{i=1}^{n+1} \xi_{T+1}(i) x_{T}(i) ]$
we solve the problem
\begin{equation}\label{eq1dp}
\mathfrak{Q}_T \left( x_{T-1}, \xi_T \right)=
\left\{
\begin{array}{l}
\sup \; \mathcal{Q}_{T+1}( x_T )  \\
(x_T, y_T, z_T) \in X_T(x_{T-1}, \xi_T),
\end{array}
\right.
\end{equation}
while at stage $t=T-1,\dots,1$, we solve
\begin{equation}\label{eq2dp}
\mathfrak{Q}_t\left( x_{t-1}, \xi_t  \right)=
\left\{
\begin{array}{l}
\sup  \;  Q_{t+1}\left( x_{t} \right) \\
(x_t, y_t, z_t) \in X_t(x_{t-1}, \xi_t) , 
\end{array}
\right.
\end{equation}
where for $t=2,\ldots,T$, $\mathcal{Q}_t(x_{t-1})=\mathbb{E}[ \mathfrak{Q}_t\left( x_{t-1}, \xi_t  \right) ]$.
With this model, we maximize the expected return of the portfolio taking into account the transaction costs,
non-negativity constraints, and bounds imposed on the different securities.

\subsubsection{CuSMuDA for portfolio selection}

We assume that the return process $(\xi_t)$ satisfies Assumption (H1).\footnote{It is possible (at the expense of the
computational time) to incorporate stagewise dependant returns within the decomposition algorithms under consideration, for instance
including in the state vector the relevant history of the returns as in \cite{morton, guiguescoap2013}.}
In this setting, we can solve the portfolio problem under consideration using
MuDA and CuSMuDA. For the sake of completeness, we show how to apply MuDA to this problem,
including the stopping criterion (CuSMuDA follows, incorporating one of the pseudo-codes from Figure \ref{figurecut1}). 
In this implementation, $N$ independent scenarios ${\tilde \xi}^k, k=(i-1)N+1,\ldots,iN$, of $(\xi_1,\xi_2,\ldots,\xi_T)$
are sampled in the forward pass of iteration $i$ to obtain $N$ sets of trial points (note that the convergence proof of Theorem \ref{convproof}
still applies for this variant of CuSMuDA). At the end of iteration $i$, we end up with approximate functions
$\mathfrak{Q}_t^i (x_{t-1} , \xi_{t j} ) = \max_{1 \leq \ell \leq iN} \langle \beta_{t j}^\ell ,  x_{t-1} \rangle$ for
$\mathfrak{Q}_t(\cdot, \xi_{t j})$  (observe that for this problem the cuts have no intercept). \\

\par {\textbf{Forward pass of iteration $i$.}} We generate $N$ scenarios 
${\tilde \xi}^k = (\tilde \xi_1^k,, \tilde \xi_2^k,\ldots,\tilde \xi_T^k), k=(i-1)N+1,\ldots,iN$, of $(\xi_2,\ldots,\xi_T)$
and solve for $k=(i-1)N+1,\ldots,iN$, and $t=1,\ldots,T-1$, the problem
$$
\left\{
\begin{array}{l}
\displaystyle \inf_{x_t, f} \; \sum_{\ell =1}^{M_{t+1}} p_{t+1 \ell} f_{\ell}\\
x_t \in X_t ( x_{t-1}^k , {\tilde \xi}_{t}^k  ),\\
f_{\ell} \geq \langle \beta_{t+1 \ell}^{m} ,  x _t\rangle, m=1,\ldots,(i-1)N, \ell =1, \ldots, M_{t+1},
\end{array}
\right.
$$
starting from $(x_{0}^k , {\tilde \xi}_1^k ) = (x_{0} , \xi_1 )$.\footnote{We use minimization instead of maximization subproblems. In this context, the optimal mean income is the opposite of the optimal value of the first stage problem.}
Let $x_{t}^k$ be an optimal solution.
%For $t=T$ and $k=(i-1)N+1,\ldots,iN$, we solve 
%$$
%\left\{
%\begin{array}{l}
%\inf  \;- \langle  x_T ,   \mathbb{E}[\xi_{T+1}] \rangle   \\
%x_T \in X_{T}( x_{T-1}^k , {\tilde \xi}_{T}^k  ),
%\end{array}
%\right.
%$$
%with optimal solution $x_T^k$.

We then sample $S$ independant scenarios of returns and simulate the policy obtained in the end of the 
forward pass on these scenarios with corresponding decisions $(\bar x_1^k,\ldots,\bar x_T^k)$ on scenario $k=1,\ldots,S$.
We
compute 
the empirical mean ${\overline{\tt{Cost}}}^i$ and standard deviation $\sigma^i$ of the
cost on these sampled scenarios:
\begin{equation} \label{formulastopping}
{\overline{\tt{Cost}}}^i = -\frac{1}{S}\sum_{k=1}^{S} \langle \mathbb{E}[\xi_{T+1}] ,  {\bar x}_T^k \rangle, \;\;
\sigma^i = \sqrt{\frac{1}{S} \sum_{k=1}^{S} \Big(-\langle \mathbb{E}[\xi_{T+1}] ,   {\bar x}_T^k \rangle  - {\overline{\tt{Cost}}}^i \Big)^2}. 
\end{equation}
This allows us to compute the upper end $z_{\sup}^i$ of a one-sided confidence interval on the
mean cost of the policy obtained at iteration $i$ given by
\begin{equation}\label{formulazsup}
z_{\sup}^i = {\overline{\tt{Cost}}}^i + \frac{\sigma^i}{\sqrt{S}} \Phi^{-1}(1-\alpha)
\end{equation}
where $\Phi^{-1}(1-\alpha)$ is the $(1-\alpha)$-quantile of the standard Gaussian distribution.\\
\par {\textbf{Backward pass of iteration $i$.}} For $k=(i-1)N+1,\ldots,iN$, we solve for $t=T$, $j=1,\ldots,M_T$,
\begin{equation}\label{back1equ}
\left\{
\begin{array}{l}
\inf  \;- \langle  x_T ,  \mathbb{E}[\xi_{T+1}] \rangle \\
x_T \in X_T ( x_{T-1}^k , \xi_{T j}  )
\end{array}
\right.
\end{equation}
and for $t=T-1,\ldots,2$, $j=1,\ldots,M_t$,
\begin{equation}\label{back2equ}
\left\{
\begin{array}{l}
\displaystyle \inf_{x_t, f}  \; \sum_{\ell=1}^{M_{t+1}} p_{t+1 \ell} f_{\ell} \\
x_t \in X_t( x_{t-1}^k , \xi_{t j}  ),\\
f_{\ell} \geq \langle \beta_{t+1 \ell}^{m} ,  x _t \rangle, m=1,\ldots,iN, \ell=1,\ldots,M_{t+1}.
\end{array}
\right.
\end{equation}
For stage $t$ problem above with realization $\xi_{t j}$ of $\xi_t$ (problem \eqref{back1equ} for $t=T$ and \eqref{back2equ} for $t<T$), let
$\lambda_{t j}^k$ be the optimal Lagrange multipliers associated to the equality constraints
and let $\mu_{t j}^k \geq 0$ be the optimal Lagrange multipliers associated with constraint
$x_t(i) \leq u(i)  \sum\limits_{\ell=1}^{n+1} \xi_{t j}(\ell) x_{t-1}(\ell)$. We compute
$$
\beta_{t j}^k = \Big( \lambda_{t j}^k - \langle u, \mu_{t j}^k \rangle {\textbf{e}} \Big) \circ \xi_{t j},
$$
where ${\textbf{e}}$ is a vector in $\mathbb{R}^{n+1}$ of ones and 
where for vectors $x, y$, the vector $x \circ y$ has components $(x \circ y)(i)=x(i) y(i)$.\\  
\par {\textbf{Stopping criterion (see \cite{shapsddp}).}}
At the end of the backward pass of iteration $i$, we solve
$$
\left\{
\begin{array}{l}
\displaystyle \inf_{x_1, f}  \;  \sum_{\ell=1}^{M_2} p_{2 \ell} f_{\ell}\\
x_1 \in X_1( x_{0} , \xi_1  ),\\
f_{\ell} \geq
\langle \beta_{2 \ell}^{m} ,  x _1\rangle, m=1,\ldots,iN, \ell=1,\ldots,M_2,
\end{array}
\right.
$$
whose optimal value provides a lower bound $z_{\inf}^i$ on the optimal value
$\mathcal{Q}_1( x_0 )$ of the problem. 
Given a tolerance $\varepsilon>0$, 
the algorithm stops either when $z_{\inf}^i=0$ and $z_{\sup}^i \leq \varepsilon$ or when
\begin{equation}\label{stoppingcriterion} 
\left| z_{\sup}^i - z_{\inf}^i   \right| \leq \varepsilon \max(1, |z_{\sup}^i |  ).
\end{equation}
In the expression above, we use  $\varepsilon \max(1, |z_{\sup}^i |  )$ instead of $\varepsilon |z_{\sup}^i |$ in the right-hand side to account for the case
$z_{\sup}^i=0$.

\subsubsection{Numerical results}

\par {\textbf{Problem data.}} We compare six methods to solve the porfolio problem presented in Section \ref{portfoliomodel}:
MuDA with sampling that we have just described (denoted by {\tt{MuDA}} for short), SDDP, CuSMuDA and SDDP 
with Level 1 cut selection (denoted by {\tt{CuSMuDA CS 1}} and {\tt{SDDP CS 1}} respectively for short), 
and CuSMuDA and SDDP with Limited Memory Level 1 cut selection  
(denoted by {\tt{CuSMuDA CS 2}} and {\tt{SDDP CS 2}} for short). 
The implementation was done in Matlab run on a laptop with Intel(R) Core(TM) i7-4510U CPU @ 2.00GHz.
All subproblems in the forward and backward passes were solved numerically using Mosek Optimization Toolbox \cite{mosek}.

We fix
$u(i)=1, i=1,\ldots,n$, while $x_0$ has components uniformly distributed in $[0,10]$.
For the stopping criterion, we use $\alpha=0.025$,
$\varepsilon=0.1$, and test two values of $N$, namely $N=1$ and $N=200$ in \eqref{stoppingcriterion}.
Below, we generate various instances of the portfolio problem as follows.
For fixed $T$ (number of stages [days for our experiment]) and $n$ (number of risky assets),
the distributions of $\xi_t(1:n)$ have $M_t=M$ realizations 
with $p_{t i}=\mathbb{P}(\xi_t = \xi_{t i})=1/M$, and
$\xi_1(1:n), \xi_{t 1}(1:n), \ldots, \xi_{t M}(1:n)$ chosen randomly among historical data of
daily returns of $n$ of the assets of the S\&P 500 index for the period 18/5/2009-28/5/2015.
These $n$ assets correspond 
to the first $n$ stocks listed in our matrix of stock prices downloaded from Wharton Research Data 
Services (WRDS: {\url{https://wrds-web.wharton.upenn.edu/wrds/}}).
The daily return $\xi_t(n+1)$ of the risk-free asset is $0.1$\% for all $t$.

Transaction costs are assumed to be known with
$\nu_t(i)=\mu_t(i)$
obtained sampling from the distribution of the random variable 
$0.08+0.06\cos(\frac{2\pi}{T} U_T )$ where $U_T$ is a random variable
with a discrete distribution over the set of integers $\{1,2,\ldots,T\}$.\\

\par {\textbf{Results.}} The computational time and number of iterations required   for solving 
22 instances of the portfolio problem for several values of
parameters $(M,T,n,N,S)$
is given in Tables \ref{tablerunningtime1}
and \ref{tablerunningtime2}.

On all instances except Instance 15 (with $M=2, T=6, n=10$),
{\tt{MuDA}} is much slower than {\tt{SDDP}}, i.e., needs 
much more time to satisfy the stopping criterion.\footnote{We considered in particular instances where $M$ is less than 
the number $2n+1$ of constraints of the subproblems to give 
MuDA a chance (according to \cite{birge1996}, cases where MuDA may be competitive with (single cut)
SDDP satisfy this requirement).}
On most instances, both variants with cut selection are quicker than their counterpart
without cut selection (on 17 instances out of 22 for SDDP and 12 out of 22 for
MuDA) and LML 1 is more efficient than Level 1 (on 18 out of 22 instances for SDDP and
19 out of 22 instances for MuDA). More precisely, both for SDDP and MuDA, we observe three different patterns:
\begin{itemize}
\item[(P1)] Instances where variants with cut selection require comparable computational bulk 
but are quicker than their counterpart without cut selection.
There are 9 of these instances for SDDP (Instances 1, 4, 7, 8, 9, 13, 15, 16, and 17)
and 8 of these instances for MuDA (Instances 1, 2, 4, 6, 13, 14, 15, and 16).
\item[(P2)] Instances where one variant with cut selection (in general LML 1) is much quicker 
than the other one, both of them being quicker than their counterpart without cut selection.
There are 8 of these instances for SDDP (Instances 2, 5, 6, 10, 11, 12, 14, 18)
and 4 of these instances for MuDA (Instances 8, 9, 10, 11).
\item[(P3)] Instances where at least one variant with cut selection 
(in  general Level 1) is much slower than its counterpart without cut selection.
There are 5 of these instances for SDDP (Instances 3, 19, 20, 21, 22) and 10 of these instances
for MuDA (3, 5, 7, 12, 17, 18, 19, 20, 21, 22).
\end{itemize}

To understand the impact of the cut selection strategies on the computational time,
it is useful to analyze the number of iterations and the 
proportion of cuts selected along the iterations of the algorithm
by the variants with cut selection in cases (P1), (P2), and (P3) described above.

As examples of instances of type (P1), consider Instance 1 for SDDP and MuDA and
Instance 2 for MuDA.

{\small{
\begin{table}[H]
\centering
\begin{tabular}{|c|}
\hline
{\tt{Instance 1:}} $M=3$, $T=48$, $n=50$, $N=1$, $S=200$\\
\hline
\end{tabular}
\begin{tabular}{|c||c|c|c|c|c|c|}
\hline
& {\tt{SDDP}} & {\tt{SDDP CS 1}}  & {\tt{SDDP CS 2}} & {\tt{MuDA}} & {\tt{CuSMuDA CS 1}}  & {\tt{CuSMuDA CS 2}} \\
\hline
CPU time & 8 750 &  2 857    & 3 201  &  22 857  &  6 800   & 6 964 \\
\hline
Iteration & 1152  &   1168 & 1286 & 1205 &1176  & 1203  \\
\hline
\end{tabular}
\begin{tabular}{|c|}
\hline
{\tt{Instance 2:}} $M=50$, $T=4$, $n=500$, $N=1$, $S=50$\\
\hline
\end{tabular}
\begin{tabular}{|c||c|c|c|c|c|c|}
\hline
& {\tt{SDDP}} & {\tt{SDDP CS 1}}  & {\tt{SDDP CS 2}} & {\tt{MuDA}} & {\tt{CuSMuDA CS 1}}  & {\tt{CuSMuDA CS 2}} \\
\hline
CPU time & 1726.1 &    727.0  &  517.1 & 3 779.0   &  3 405   & 3 348 \\
\hline
Iteration & 200  &  200  & 200 & 41 & 42 & 50  \\
\hline
\end{tabular}
\begin{tabular}{|c|}
\hline
{\tt{Instance 3:}} $M=5$, $T=12$, $n=5$, $N=200$, $S=200$\\
\hline
\end{tabular}
\begin{tabular}{|c||c|c|c|c|c|c|}
\hline
& {\tt{SDDP}} & {\tt{SDDP CS 1}}  & {\tt{SDDP CS 2}} & {\tt{MuDA}} & {\tt{CuSMuDA CS 1}}  & {\tt{CuSMuDA CS 2}} \\
\hline
CPU time &  506.4  &  1 378.2    & 193.2   &   799.8 &   3053.5   & 168.9  \\
\hline
Iteration &  9 &  11   & 7  & 6   &   8   & 7\\
\hline
\end{tabular}
\begin{tabular}{|c|}
\hline
{\tt{Instance 4:}} $M=2$, $T=6$, $n=200$, $N=1$, $S=50$\\
\hline
\end{tabular}
\begin{tabular}{|c||c|c|c|c|c|c|}
\hline
& {\tt{SDDP}} & {\tt{SDDP CS 1}}  & {\tt{SDDP CS 2}} & {\tt{MuDA}} & {\tt{CuSMuDA CS 1}}  & {\tt{CuSMuDA CS 2}} \\
\hline
CPU time & 25.5 & 22.2    &21.0 &  43.3  & 34.4    & 34.4 \\
\hline
Iteration &  83 & 97   & 94 & 82 & 93 &  92 \\
\hline
\end{tabular}
\begin{tabular}{|c|}
\hline
{\tt{Instance 5:}} $M=5$, $T=12$, $n=5$, $N=1$, $S=200$\\
\hline
\end{tabular}
\begin{tabular}{|c||c|c|c|c|c|c|}
\hline
& {\tt{SDDP}} & {\tt{SDDP CS 1}}  & {\tt{SDDP CS 2}} & {\tt{MuDA}} & {\tt{CuSMuDA CS 1}}  & {\tt{CuSMuDA CS 2}} \\
\hline
CPU time & 22.0   &  21.9    & 16.9   &  42.2  & 104.5     & 24.0   \\
\hline
Iteration &  200 & 200     &  200 &  200  & 200     & 200\\
\hline
\end{tabular}
\begin{tabular}{|c|}
\hline
{\tt{Instance 6:}} $M=5$, $T=12$, $n=10$, $N=1$, $S=200$\\
\hline
\end{tabular}
\begin{tabular}{|c||c|c|c|c|c|c|}
\hline
& {\tt{SDDP}} & {\tt{SDDP CS 1}}  & {\tt{SDDP CS 2}} & {\tt{MuDA}} & {\tt{CuSMuDA CS 1}}  & {\tt{CuSMuDA CS 2}} \\
\hline
CPU time & 34.1   &  25.7    & 19.5   &  76.6  & 43.8     & 32.7   \\
\hline
Iteration &  200 & 200     &  200 &  200  & 200     & 200\\
\hline
\end{tabular}
\begin{tabular}{|c|}
\hline
{\tt{Instance 7:}} $M=10$, $T=6$, $n=10$, $N=1$, $S=200$\\
\hline
\end{tabular}
\begin{tabular}{|c||c|c|c|c|c|c|}
\hline
& {\tt{SDDP}} & {\tt{SDDP CS 1}}  & {\tt{SDDP CS 2}} & {\tt{MuDA}} & {\tt{CuSMuDA CS 1}}  & {\tt{CuSMuDA CS 2}} \\
\hline
CPU time & 53.7   &  17.3    & 14.8   &  90.3  & 143.0     & 26.2   \\
\hline
Iteration &  200 & 200     &  200 &  200  & 200     & 200\\
\hline
\end{tabular}
\begin{tabular}{|c|}
\hline
{\tt{Instance 8:}} $M=20$, $T=6$, $n=5$, $N=1$, $S=200$\\
\hline
\end{tabular}
\begin{tabular}{|c||c|c|c|c|c|c|}
\hline
& {\tt{SDDP}} & {\tt{SDDP CS 1}}  & {\tt{SDDP CS 2}} & {\tt{MuDA}} & {\tt{CuSMuDA CS 1}}  & {\tt{CuSMuDA CS 2}} \\
\hline
CPU time & 33.7   &  26.5    & 23.9   &  214.5  & 192.3     & 50.3   \\
\hline
Iteration &  200 & 200     &  200 &  200  & 200     & 200\\
\hline
\end{tabular}
\begin{tabular}{|c|}
\hline
{\tt{Instance 9:}} $M=50$, $T=12$, $n=10$, $N=1$, $S=200$\\
\hline
\end{tabular}
\begin{tabular}{|c||c|c|c|c|c|c|}
\hline
& {\tt{SDDP}} & {\tt{SDDP CS 1}}  & {\tt{SDDP CS 2}} & {\tt{MuDA}} & {\tt{CuSMuDA CS 1}}  & {\tt{CuSMuDA CS 2}} \\
\hline
CPU time & 340.7   &  181.8    & 158.8   &  3 649.4  & 2 541.8     & 835.3   \\
\hline
Iteration &  200 & 200     &  200 &  179  & 200     & 200\\
\hline
\end{tabular}
\begin{tabular}{|c|}
\hline
{\tt{Instance 10:}} $M=10$, $T=6$, $n=10$, $N=200$, $S=200$\\
\hline
\end{tabular}
\begin{tabular}{|c||c|c|c|c|c|c|}
\hline
& {\tt{SDDP}} & {\tt{SDDP CS 1}}  & {\tt{SDDP CS 2}} & {\tt{MuDA}} & {\tt{CuSMuDA CS 1}}  & {\tt{CuSMuDA CS 2}} \\
\hline
CPU time & 256.3   &  208.3    & 84.8   &  2 321.5  & 1 185.4     & 191.8  \\
\hline
Iteration &  5 & 5    &  5 &  6  & 5     & 5\\
\hline
\end{tabular}
\begin{tabular}{|c|}
\hline
{\tt{Instance 11:}} $M=20$, $T=6$, $n=5$, $N=200$, $S=200$\\
\hline
\end{tabular}
\begin{tabular}{|c||c|c|c|c|c|c|}
\hline
& {\tt{SDDP}} & {\tt{SDDP CS 1}}  & {\tt{SDDP CS 2}} & {\tt{MuDA}} & {\tt{CuSMuDA CS 1}}  & {\tt{CuSMuDA CS 2}} \\
\hline
CPU time & 250.9   &  189.1    & 72.0   &  2 721.3  & 1 838.3     & 232.6 \\
\hline
Iteration &  4 & 4    &  3 &  4  & 4     & 4\\
\hline
\end{tabular}
\caption{Computational time (in seconds) and number of iterations for solving instances of the portfolio problem of Section \ref{portfoliomodel} with 
{\tt{SDDP}}, {\tt{SDDP CS 1}}, {\tt{SDDP CS 2}}, {\tt{MuDA}}, {\tt{CuSMuDA CS 1}}, and {\tt{CuSMuDA CS 2}}.}
\label{tablerunningtime1}
\end{table}
}}

%%%%%%%%%%%%%%%%%%%%%%%%%%%%%%%%%%%%%%%%%%%%%%%%%%%%%%%%%%%%%%%%%%%%%%%%%%%%%%%%%%%%%%%%%5

{\small{
\begin{table}[H]
\centering
\begin{tabular}{|c|}
\hline
{\tt{Instance 12:}} $M=2$, $T=6$, $n=450$, $N=1$, $S=50$\\
\hline
\end{tabular}
\begin{tabular}{|c||c|c|c|c|c|c|}
\hline
& {\tt{SDDP}} & {\tt{SDDP CS 1}}  & {\tt{SDDP CS 2}} & {\tt{MuDA}} & {\tt{CuSMuDA CS 1}}  & {\tt{CuSMuDA CS 2}} \\
\hline
CPU time & 175.3 & 117.7    &153.7 &  235.6  & 351.7    & 183.7 \\
\hline
Iteration &  158 & 154   &173 & 128 & 181 &  168 \\
\hline
\end{tabular}
\begin{tabular}{|c|}
\hline
{\tt{Instance 13:}} $M=2$, $T=12$, $n=50$, $N=1$, $S=50$\\
\hline
\end{tabular}
\begin{tabular}{|c||c|c|c|c|c|c|}
\hline
& {\tt{SDDP}} & {\tt{SDDP CS 1}}  & {\tt{SDDP CS 2}} & {\tt{MuDA}} & {\tt{CuSMuDA CS 1}}  & {\tt{CuSMuDA CS 2}} \\
\hline
CPU time & 46.3 & 21.5    & 21.1 &  103.2  & 31.7     & 36.8 \\
\hline
Iteration &  122 & 137   &132 & 125 & 127 &  137 \\
\hline
\end{tabular}
\begin{tabular}{|c|}
\hline
{\tt{Instance 14:}} $M=2$, $T=6$, $n=10$, $N=1$, $S=10$\\
\hline
\end{tabular}
\begin{tabular}{|c||c|c|c|c|c|c|}
\hline
& {\tt{SDDP}} & {\tt{SDDP CS 1}}  & {\tt{SDDP CS 2}} & {\tt{MuDA}} & {\tt{CuSMuDA CS 1}}  & {\tt{CuSMuDA CS 2}} \\
\hline
CPU time & 0.52 & 0.35    & 0.45 &  0.38  & 0.37     & 0.34 \\
\hline
Iteration &  13 & 13   &13&11&13&13 \\
\hline
\end{tabular}
\begin{tabular}{|c|}
\hline
{\tt{Instance 15:}} $M=3$, $T=36$, $n=30$, $N=1$, $S=200$\\
\hline
\end{tabular}
\begin{tabular}{|c||c|c|c|c|c|c|}
\hline
& {\tt{SDDP}} & {\tt{SDDP CS 1}}  & {\tt{SDDP CS 2}} & {\tt{MuDA}} & {\tt{CuSMuDA CS 1}}  & {\tt{CuSMuDA CS 2}} \\
\hline
CPU time & 1029   &  341    & 394   &  1 710  & 704     & 634 \\
\hline
Iteration &  454 & 467    & 489 &  436  & 463     & 462\\
\hline
\end{tabular}
\begin{tabular}{|c|}
\hline
{\tt{Instance 16:}} $M=3$, $T=48$, $n=10$, $N=1$, $S=50$\\
\hline
\end{tabular}
\begin{tabular}{|c||c|c|c|c|c|c|}
\hline
& {\tt{SDDP}} & {\tt{SDDP CS 1}}  & {\tt{SDDP CS 2}} & {\tt{MuDA}} & {\tt{CuSMuDA CS 1}}  & {\tt{CuSMuDA CS 2}} \\
\hline
CPU time & 56.8   &  44.2    & 39.3   &  75.4  & 46.1     & 39.3 \\
\hline
Iteration &  100 & 139    &  124 &  100  & 99     & 98\\
\hline
\end{tabular}
\begin{tabular}{|c|}
\hline
{\tt{Instance 17:}} $M=20$, $T=5$, $n=4$, $N=200$, $S=200$\\
\hline
\end{tabular}
\begin{tabular}{|c||c|c|c|c|c|c|}
\hline
& {\tt{SDDP}} & {\tt{SDDP CS 1}}  & {\tt{SDDP CS 2}} & {\tt{MuDA}} & {\tt{CuSMuDA CS 1}}  & {\tt{CuSMuDA CS 2}} \\
\hline
CPU time &  103.6  &  88.2    &   81.2 & 538.2  &   807   & 78.6 \\
\hline
Iteration &  3 & 4    & 4  &  2  &  2    & 2 \\
\hline
\end{tabular}
\begin{tabular}{|c|}
\hline
{\tt{Instance 18:}} $M=20$, $T=5$, $n=5$, $N=200$, $S=200$\\
\hline
\end{tabular}
\begin{tabular}{|c||c|c|c|c|c|c|}
\hline
& {\tt{SDDP}} & {\tt{SDDP CS 1}}  & {\tt{SDDP CS 2}} & {\tt{MuDA}} & {\tt{CuSMuDA CS 1}}  & {\tt{CuSMuDA CS 2}} \\
\hline
CPU time & 30.9  &  26.7    &  17.8  & 1021.8  &    1764.6  &  142.8\\
\hline
Iteration &  1 &   1  & 1  & 3   &   3   & 3 \\
\hline
\end{tabular}
\begin{tabular}{|c|}
\hline
{\tt{Instance 19:}} $M=20$, $T=5$, $n=6$, $N=200$, $S=200$\\
\hline
\end{tabular}
\begin{tabular}{|c||c|c|c|c|c|c|}
\hline
& {\tt{SDDP}} & {\tt{SDDP CS 1}}  & {\tt{SDDP CS 2}} & {\tt{MuDA}} & {\tt{CuSMuDA CS 1}}  & {\tt{CuSMuDA CS 2}} \\
\hline
CPU time &  132.2  &   273.2   & 127.8   &  2 752.2 &  3 853.2    & 315 \\
\hline
Iteration &  3 &   5  &  5 & 5   &  5    & 5 \\
\hline
\end{tabular}
\begin{tabular}{|c|}
\hline
{\tt{Instance 20:}} $M=10$, $T=8$, $n=4$, $N=200$, $S=200$\\
\hline
\end{tabular}
\begin{tabular}{|c||c|c|c|c|c|c|}
\hline
& {\tt{SDDP}} & {\tt{SDDP CS 1}}  & {\tt{SDDP CS 2}} & {\tt{MuDA}} & {\tt{CuSMuDA CS 1}}  & {\tt{CuSMuDA CS 2}} \\
\hline
CPU time & 134.7   &   177.9   & 67.4   & 476.4  &  976.2    & 93.6 \\
\hline
Iteration &  4 &   4  &  4 &  3  &   3   & 3 \\
\hline
\end{tabular}
\begin{tabular}{|c|}
\hline
{\tt{Instance 21:}} $M=10$, $T=8$, $n=5$, $N=200$, $S=200$\\
\hline
\end{tabular}
\begin{tabular}{|c||c|c|c|c|c|c|}
\hline
& {\tt{SDDP}} & {\tt{SDDP CS 1}}  & {\tt{SDDP CS 2}} & {\tt{MuDA}} & {\tt{CuSMuDA CS 1}}  & {\tt{CuSMuDA CS 2}} \\
\hline
CPU time &  264.7  & 297.0     & 90.9   & 2 603.4  &  4 438.8    & 258 \\
\hline
Iteration &  5 &    5 & 4  &  6  &  5    & 7 \\
\hline
\end{tabular}
\begin{tabular}{|c|}
\hline
{\tt{Instance 22:}} $M=10$, $T=8$, $n=6$, $N=200$, $S=200$\\
\hline
\end{tabular}
\begin{tabular}{|c||c|c|c|c|c|c|}
\hline
& {\tt{SDDP}} & {\tt{SDDP CS 1}}  & {\tt{SDDP CS 2}} & {\tt{MuDA}} & {\tt{CuSMuDA CS 1}}  & {\tt{CuSMuDA CS 2}} \\
\hline
CPU time &  256.7  &  312.7    & 162.1   & 3 808.2  & 6 633     &  303\\
\hline
Iteration &  5 &   5  &  6 &  7  &   9   &  6\\
\hline
\end{tabular}
\caption{Computational time (in seconds) and number of iterations for solving instances of the portfolio problem of Section \ref{portfoliomodel} with 
{\tt{SDDP}}, {\tt{SDDP CS 1}}, {\tt{SDDP CS 2}}, {\tt{MuDA}}, {\tt{CuSMuDA CS 1}}, and {\tt{CuSMuDA CS 2}}.}
\label{tablerunningtime2}
\end{table}
}}
For Instance 1, the mean proportion (over the iterations of the algorithm) of cuts selected
by all variants with cut selection, i.e., {\tt{SDDP CS 1}}, {\tt{SDDP CS 2}}, {\tt{CuSMuDA CS 1}}, and {\tt{CuSMuDA CS 2}}, 
tends to increase with the stage and is 
very low for nearly all stages (at most 0.1 until stage 30 and at most 0.2 until stage 40, knowing
that there are $T=48$ stages for that instance); see the bottom plot of Figure \ref{fig:f_1}
which represents the evolution of the mean proportion of cuts selected as a function of the stage.
This makes sense since at the last stage, the cuts for functions $\mathfrak{Q}_T(\cdot, \xi_{T j})$ are exact (and as we recalled, all of them
are selected with  {\tt{CuSMuDA CS 1}}) but not necessarily for stages $t<T$ because for these stages approximate recourse functions
are used and the approximation errors propagate backward. Therefore it is expected that at the early stages old cuts (computed at the first
iterations using crude approximations of the recourse functions) will probably not be selected.
Additionally, another reason why fewer cuts are selected in earlier stages may come from the fact that there are fewer distinct sampled points 
in the state-space.

Moreover the proportion of cuts selected is very similar for all variants.

Since the degradation in the upper and lower bound for variants with cut selection is very limited (see the evolution of the 
upper and lower bounds along iterations for all methods on the top right plot of Figure \ref{fig:f_1}), 
the number of iterations required to solve Instance 1 by these variants and their counterpart without cut selection
is similar.\footnote{We see in particular that the bounds for all algorithms are very close to each
other at the last iteration and that, as expected,
the lower bounds increase and the upper bounds 
tend to decrease along iterations.
If we do not know the optimal value $\mathcal{Q}_1( x_0)$, these observations 
(that we checked on all instances) are a good indication that 
all algorithms were correctly implemented.}
\begin{figure}[H]
%$$
\centering
\begin{tabular}{cc}
\includegraphics[scale=0.5]{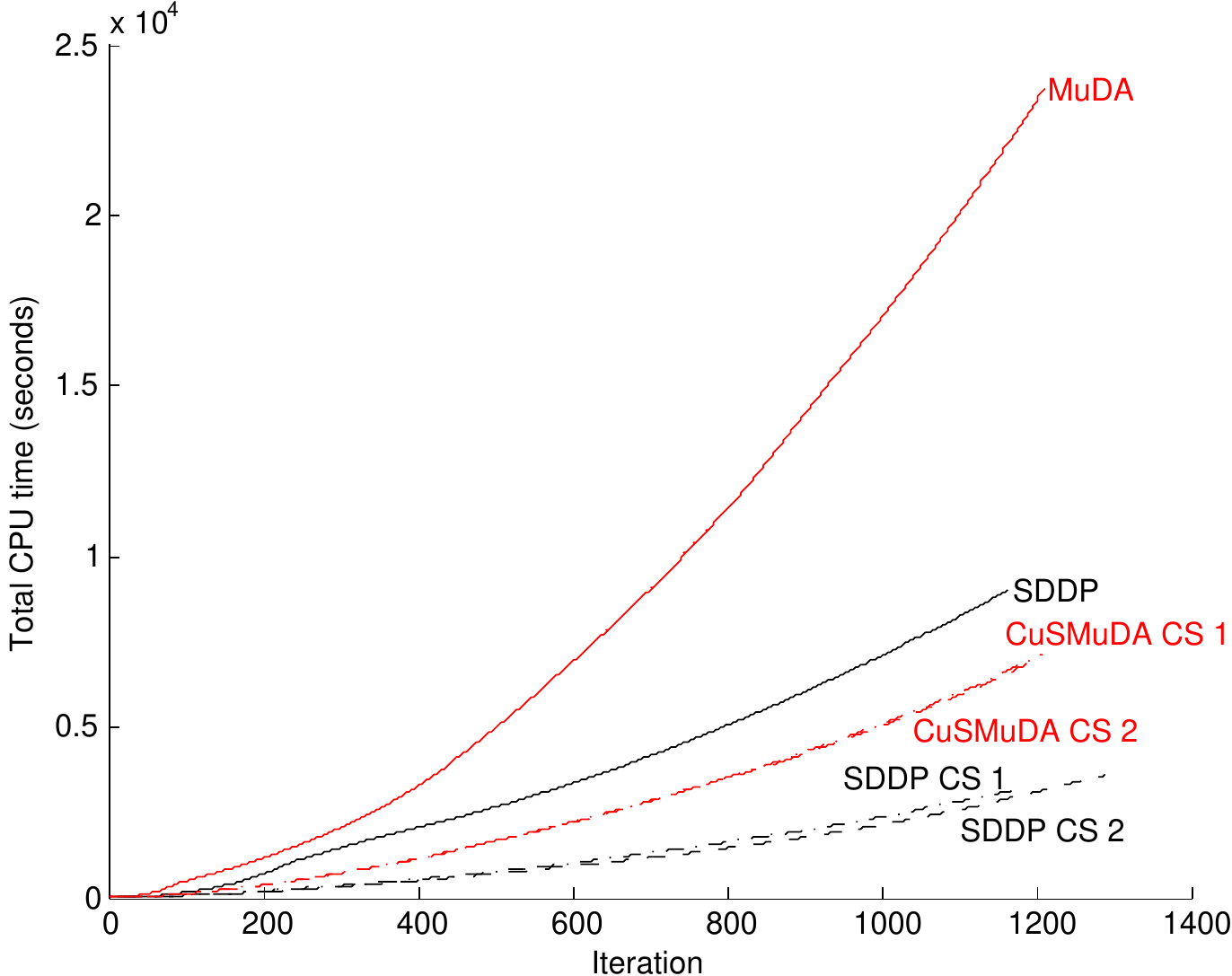}&
\includegraphics[scale=0.5]{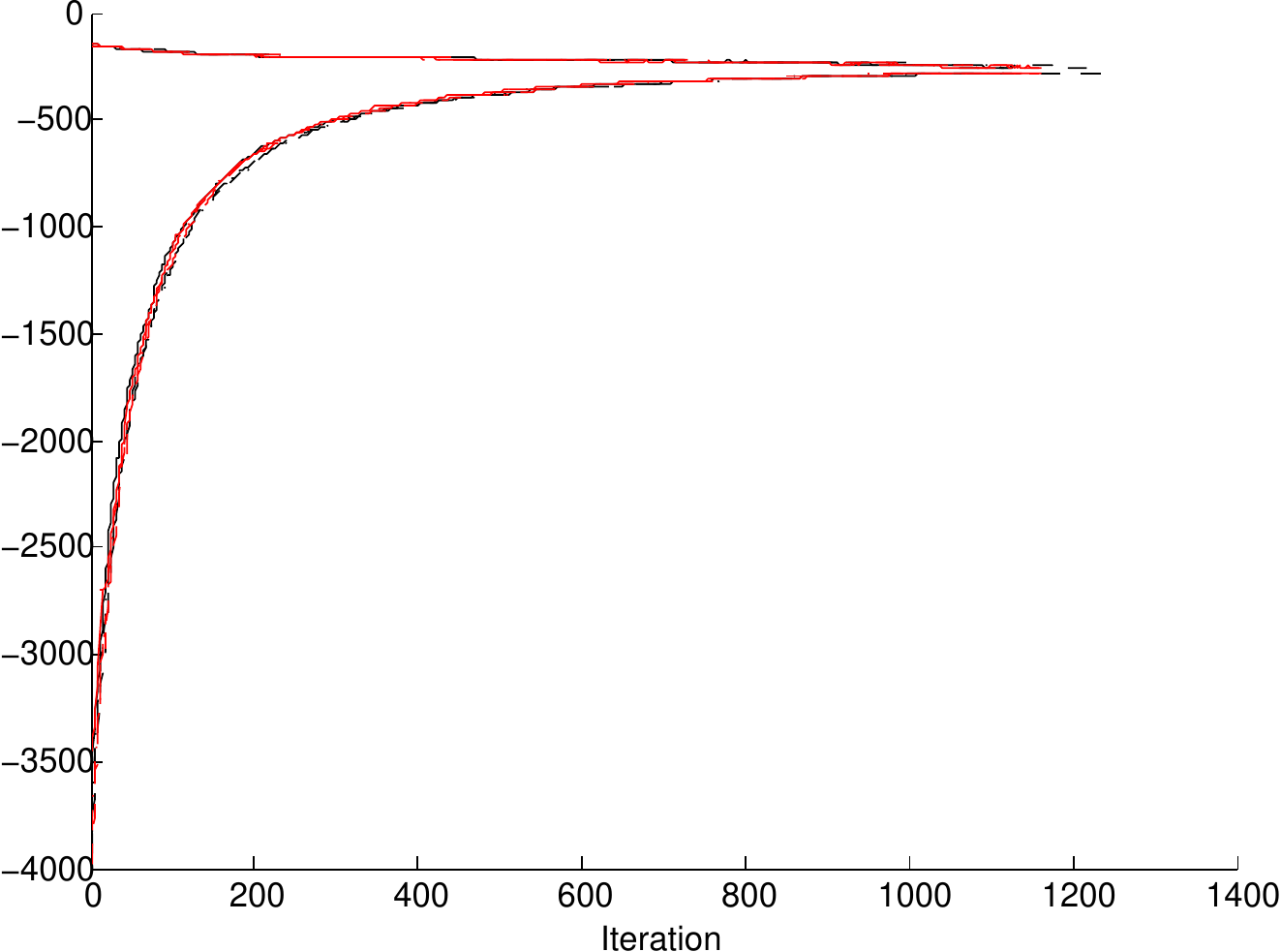}
\end{tabular}
\begin{tabular}{c}
\includegraphics[scale=0.5]{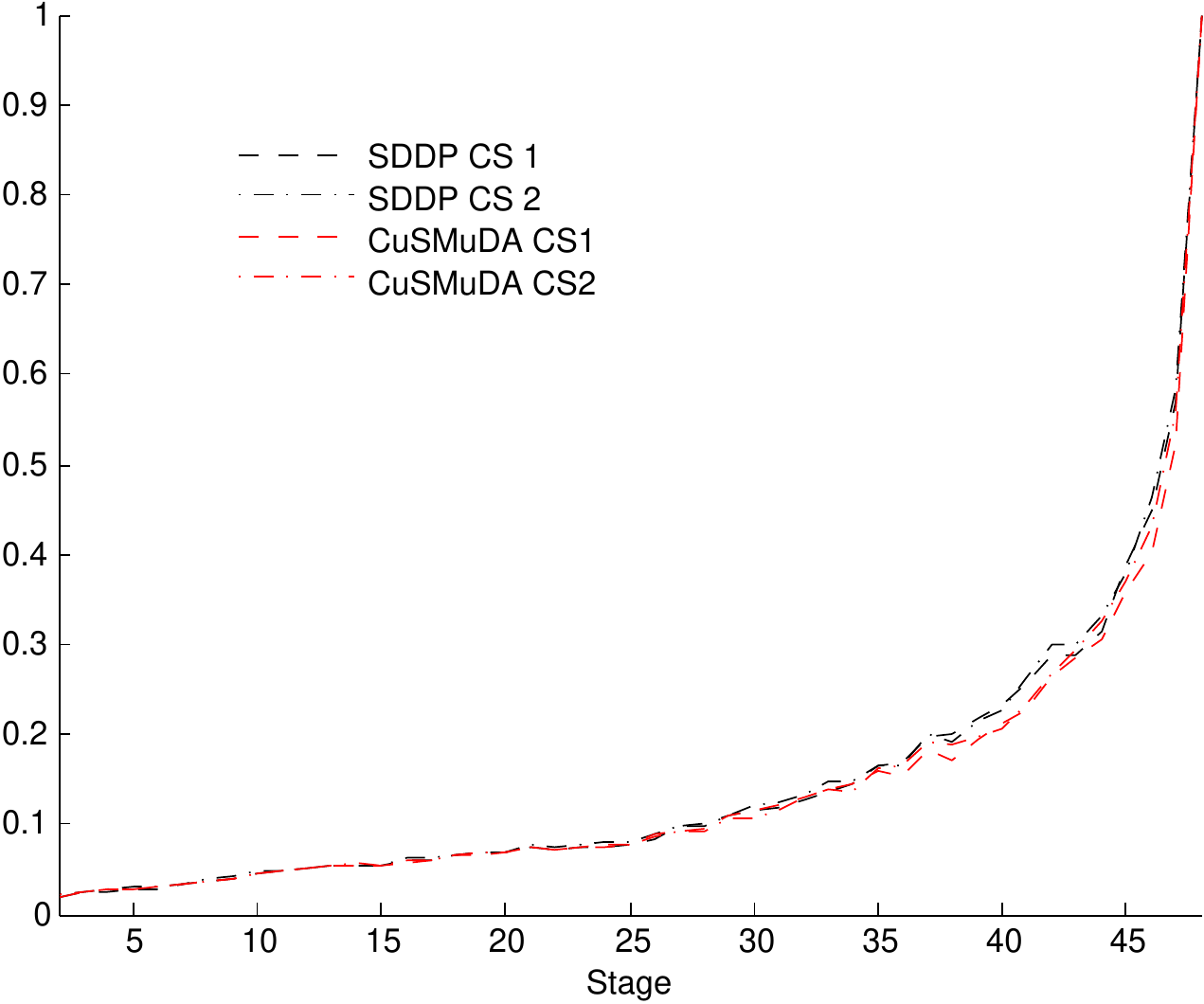}
\end{tabular}
%$$
\caption{Top left: total CPU time (in seconds) as a function of the 
number of  iterations for {\tt{SDDP}}, {\tt{SDDP CS 1}}, {\tt{SDDP CS 2}},
{\tt{MuDA}}, {\tt{CuSMuDA CS 1}}, and {\tt{CuSMuDA CS 2}}
to solve Instance 1.
Top right: evolution of the upper bounds $z_{\sup}^i$ and lower bounds $z_{\inf}^i$ along the iterations
of {\tt{SDDP}}, {\tt{SDDP CS 1}}, {\tt{SDDP CS 2}},
{\tt{MuDA}}, {\tt{CuSMuDA CS 1}}, and {\tt{CuSMuDA CS 2}}
to solve Instance 1 (no legend is given on this plot since all curves are almost identical).
Bottom: mean proportion of cuts (over the iterations of the algorithm) selected
for stages $t=2,\ldots,T=48$ by {\tt{SDDP CS 1}}, {\tt{SDDP CS 2}}, {\tt{CuSMuDA CS 1}}, and {\tt{CuSMuDA CS 2}}
to solve Instance 1.\label{fig:f_1}}
\end{figure}

Therefore, variants with cut selection are much quicker and 
the quickest variant with cut selection is the one requiring the least number of iterations; see the top left
plot of Figure \ref{fig:f_1} which plots the total CPU time as a function of the number of iterations.

For Instance 2 solved by MuDA and its variants with cut selection, we refer to Figure \ref{fig:f_2}.
We see on this figure that these variants again select a very small proportion of cuts for stages
2 and 3 (here, the number of stages is $T=4$) and that this proportion increases with the stage. The degradation in 
the lower bound for variants with cut selection
is larger than with Instance 1 but is still very limited.
Therefore, variants with cut selection are much quicker.
{\tt{CuSMuDA CS 2}} requires less cuts than {\tt{CuSMuDA CS 1}} but more iterations 
and CPU time for both variants is similar.\\

\par To explain patterns (P2) and (P3), we consider Instance 2 for SDDP, of type (P2),
Instance 3 for SDDP and MuDA, of type (P3), and refer 
to Figures \ref{fig:f_2} and \ref{fig:lastfig} which are the analogues of Figure \ref{fig:f_1} for
Instances 2 and 3. 

\begin{figure}[H]
%$$
\centering
\begin{tabular}{cc}
\includegraphics[scale=0.55]{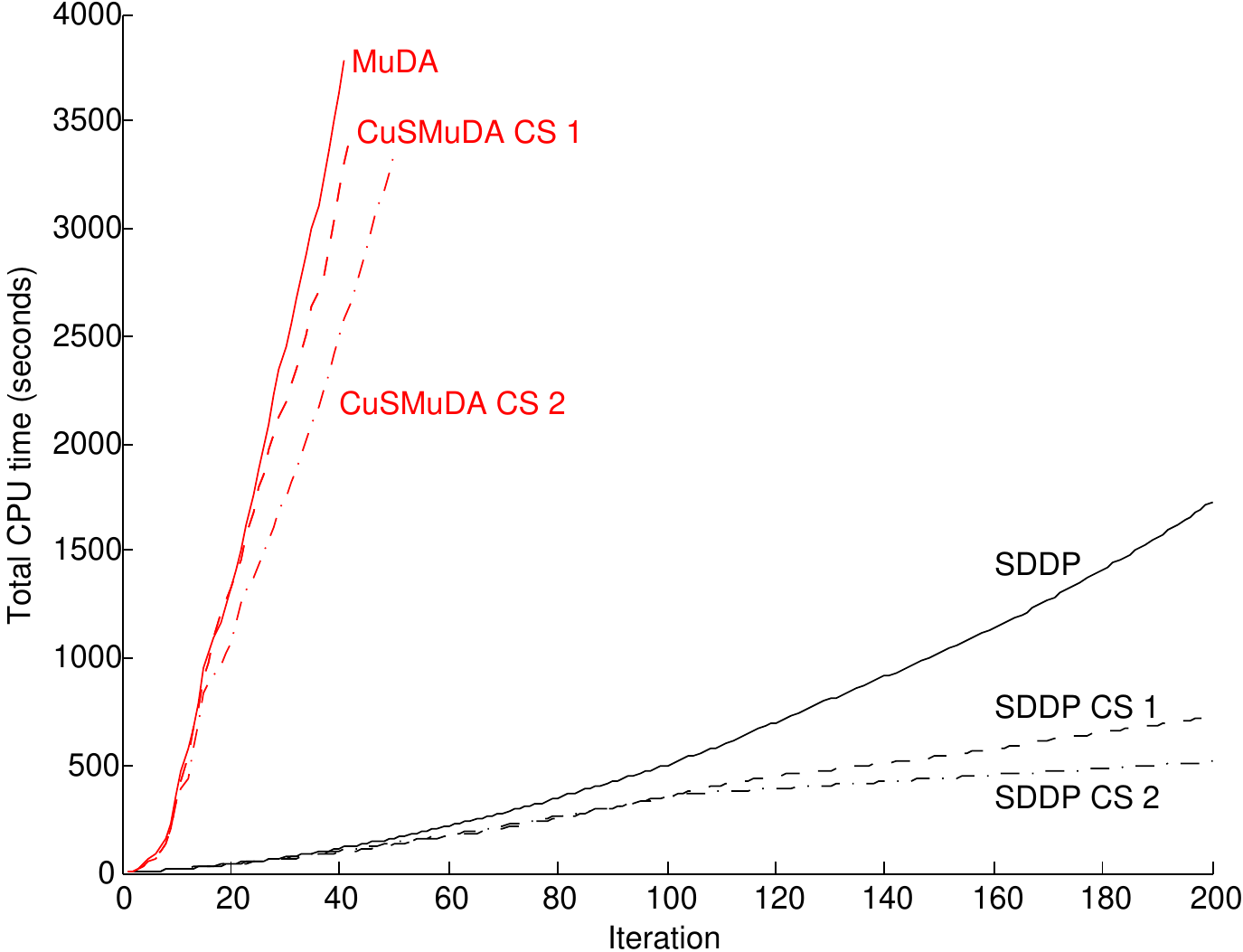}& 
\includegraphics[scale=0.55]{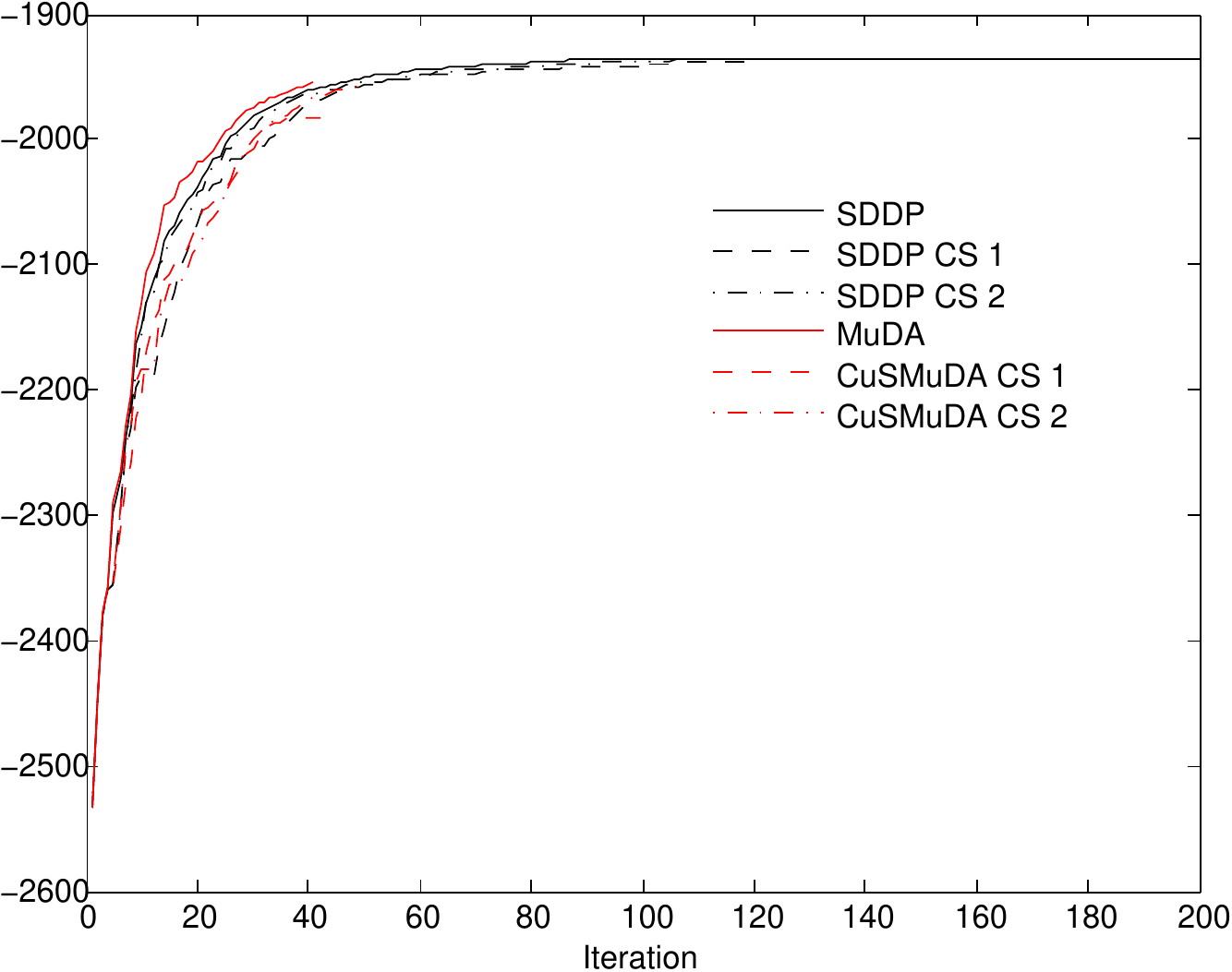}
\end{tabular}
%$$
\begin{tabular}{cc}
\includegraphics[scale=0.55]{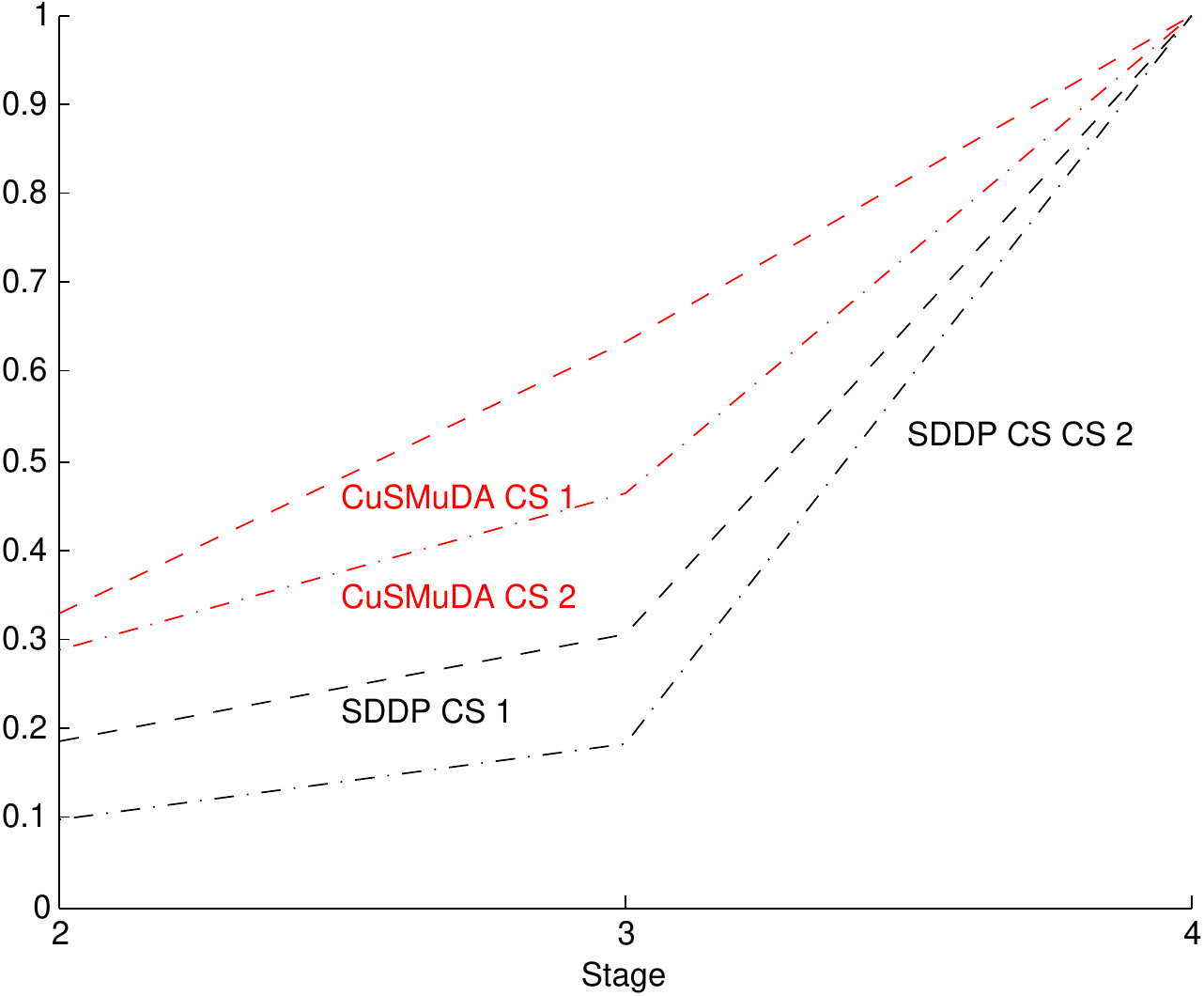}
\end{tabular}
\caption{Top left: total CPU time (in seconds) as a function of the 
number of  iterations for {\tt{SDDP}}, {\tt{SDDP CS 1}}, {\tt{SDDP CS 2}},
{\tt{MuDA}}, {\tt{CuSMuDA CS 1}}, and {\tt{CuSMuDA CS 2}}
to solve Instance 2.
Top right: evolution of the lower bounds $z_{\inf}^i$ along the iterations
of {\tt{SDDP}}, {\tt{SDDP CS 1}}, {\tt{SDDP CS 2}},
{\tt{MuDA}}, {\tt{CuSMuDA CS 1}}, and {\tt{CuSMuDA CS 2}}
to solve Instance 2.
Bottom: mean proportion of cuts (over the iterations of the algorithm) selected
for stages $t=2,3,T=4$, by {\tt{SDDP CS 1}}, {\tt{SDDP CS 2}}, {\tt{CuSMuDA CS 1}}, and {\tt{CuSMuDA CS 2}}
to solve Instance 2. \label{fig:f_2}}
\end{figure}

\begin{figure}[H]
%$$
\centering
\begin{tabular}{cc}
\includegraphics[scale=0.55]{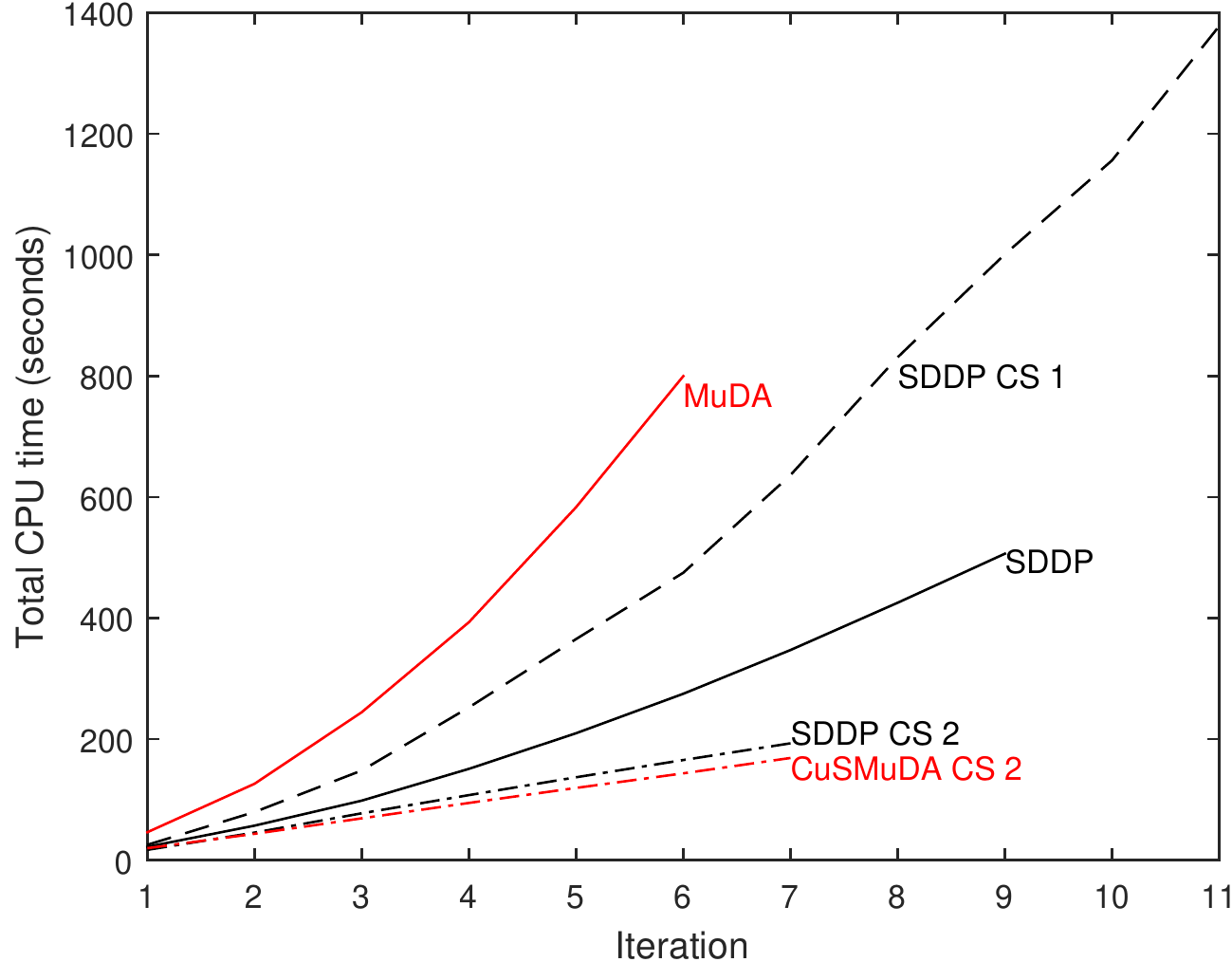}& 
\includegraphics[scale=0.55]{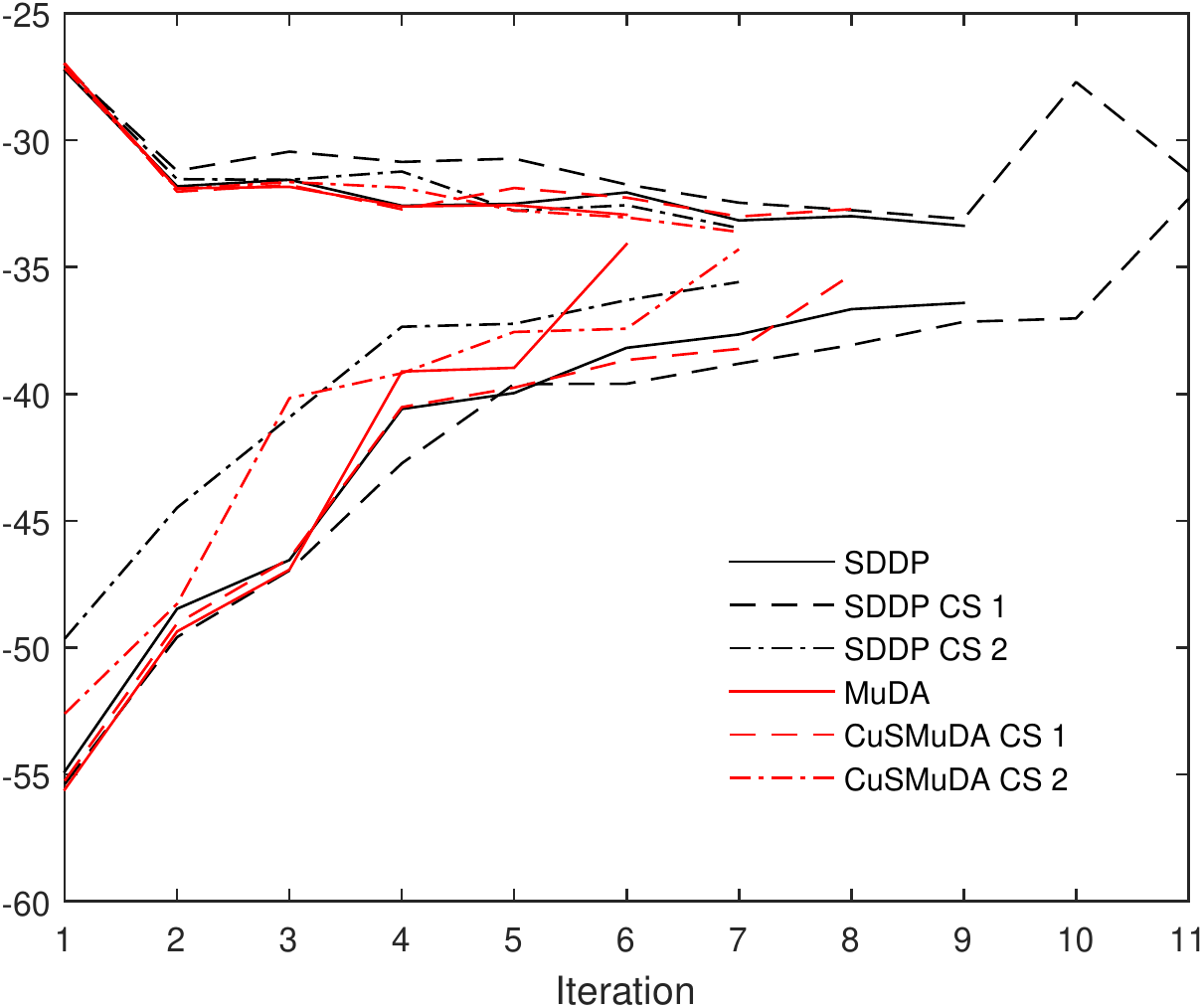}
\end{tabular}
%$$
\begin{tabular}{cc}
\includegraphics[scale=0.55]{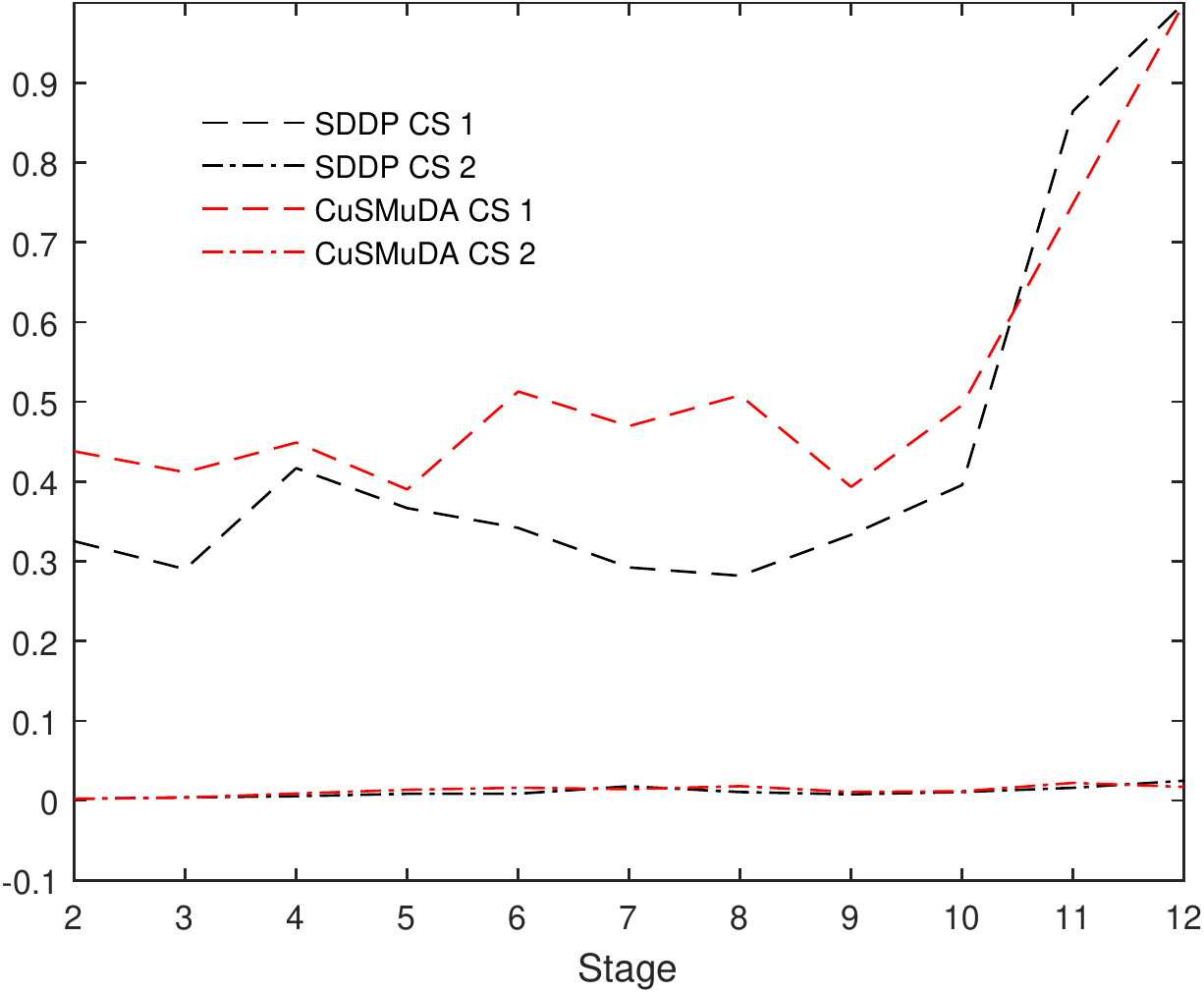}
\end{tabular}
\caption{Top left: total CPU time (in seconds) as a function of the 
number of  iterations for {\tt{SDDP}}, {\tt{SDDP CS 1}}, {\tt{SDDP CS 2}},
{\tt{MuDA}}, and {\tt{CuSMuDA CS 2}}
to solve Instance 3 (we did not represent the curve for {\tt{CuSMuDA CS 1}} since the total
CPU time with {\tt{CuSMuDA CS 1}} is much larger).
Top right: evolution of the upper bounds $z_{\sup}^i$ and lower bounds $z_{\inf}^i$ along the iterations
of {\tt{SDDP}}, {\tt{SDDP CS 1}}, {\tt{SDDP CS 2}},
{\tt{MuDA}}, {\tt{CuSMuDA CS 1}}, and {\tt{CuSMuDA CS 2}}
to solve Instance 3.
Bottom: mean proportion of cuts (over the iterations of the algorithm) selected
for stages $t=2,\ldots,T=12$, by {\tt{SDDP CS 1}}, {\tt{SDDP CS 2}}, {\tt{CuSMuDA CS 1}}, and {\tt{CuSMuDA CS 2}}
to solve Instance 3. \label{fig:lastfig}}
\end{figure}

On Figure \ref{fig:f_2}, we see that on Instance 2, {\tt{SDDP CS 1}}
and {\tt{SDDP CS 2}} select a small proportion of cuts for stages $2$ and $3$ (recall that this
is an instance with $T=4$ stages) and require, as {\tt{SDDP}}, 200 iterations (the evolution 
of the lower bound along iterations is similar for all 3 methods). Moreover, {\tt{SDDP CS 2}}
selects significantly less cuts than {\tt{SDDP CS 1}} for stages 2 and 3 and therefore is much 
quicker than {\tt{SDDP CS 1}}.
This means that a large number of cuts have the same value at some trial points and for each  such
trial point if {\tt{SDDP CS 1}} selects all these cuts {\tt{CuSMuDA CS 2}} only selects the oldest of these cuts.\\

\par Instance 3 is an example of a problem where LML 1 cut selection strategy selects a very small number of cuts 
for all stages whereas Level 1 selects many more cuts, both for MuDA and SDDP (see Figure \ref{fig:lastfig}). On top of that, 
{\tt{SDDP CS 1}} (resp. {\tt{CuSMuDA CS 1}}) needs more iterations than {\tt{SDDP}} and {\tt{SDDP CS 2}}
(resp. {\tt{MuDA}} and {\tt{CuSMuDA CS 2}}).
Therefore the CPU time needed to solve Instance 3 with {\tt{SDDP CS 1}} (resp. {\tt{CuSMuDA CS 1}})
is much larger than the CPU time needed to solve this instance with  
{\tt{SDDP}} and {\tt{SDDP CS 2}} (resp. {\tt{MuDA}} and {\tt{CuSMuDA CS 2}}).
This is an example of an instance where the time spent to select the cuts is not  compensated by the (small)
reduction in CPU time for solving the problems in the backward and forward passes.

\newpage

Summarizing our observations, 
\begin{itemize}
\item Pattern (P3) occurs when a variant with cut selection selects 
a large proportion of cuts and requires
too many iterations;
\item Patterns (P1) and (P2) occur when variants with cut selection select a small number of cuts
and do not need much more iterations than the variant without cut selection.
In this situation, (P1) occurs either when (i) both variants with cut selection select a similar number
of cuts or when (ii) one variant selects less cuts than the other but needs slightly more iterations.
\end{itemize}

One last comment is now in order. We have already observed that on all
experiments, Level 1 and Territory Algorithm cut selection strategies, i.e., {\tt{SDDP CS 1}} and {\tt{CuSMuDA CS 1}},
correctly select all cuts at the final stage. However, a crude implementation of the
Level 1 pseudo-code given in Figure \ref{figurecut1} resulted in 
the elimination of cuts at the final stage. This comes from the fact that approximate solutions of the
optimization subproblems are computed. Therefore two optimization problems could  compute the same cuts but the solver
may return two different (very close) approximate solutions. Similarly, a cut may be in theory the highest at some point
(for instance cut $\mathcal{C}_{T j}^k$ is, in theory, the highest at $x_{T-1}^k$) but numerically the value of another
cut at this point may be deemed slightly higher, because of numerical errors.
The remedy is to introduce a small error term $\varepsilon_0$ ($\varepsilon_0=10^{-6}$ in our experiments)
such that the values $V_1$ and $V_2$ of two cuts $\mathcal{C}_1$ and $\mathcal{C}_2$  at a trial point are considered equal if
$|V_2 -V_1| \leq \varepsilon_0 \max(1, |V_1|)$ while $\mathcal{C}_1$ is above $\mathcal{C}_2$ at this trial point if
$V_1 \geq  V_2 + \varepsilon_0 \max(1, |V_1|)$.
Therefore the pseudo-codes of Level 1 and LML 1 
given in Figure \ref{figurecut1} need to be modified. The corresponding pseudo-codes taking into account the approximation
errors are given in Figure \ref{figurecut2} in the Appendix.
Adding cuts that improve the approximation by more than some threshold $\varepsilon_0$ was also used in \cite{lohndorf}.

\subsection{Inventory management}

\subsubsection{Model} \label{sec:inventory}

We consider the $T$-stage inventory management problem given in (Shapiro et al. 2009).
For each stage $t=1,\ldots, T$, on the basis of the inventory
level $x_{t-1}$ at the beginning of period $t$, we have to decide on the quantity $y_t - x_{t-1}$ of a product
to buy so that the inventory level becomes $y_t$. 
Given demand $\xi_t$ for that product 
for stage $t$, the  inventory level is $x_{t}=y_t-\xi_t$ at the beginning of stage $t+1$.
The inventory level can become negative, in which case a backorder cost is paid.
If one is interested in minimizing the average cost over the optimization period,
we need to solve the following dynamic programming equations: for $t=1,\ldots,T$, defining
$\mathcal{Q}_t( x_{t-1} ) = \mathbb{E}_{\xi_t}[\mathfrak{Q}_t( x_{t-1}, \xi_t)]$, the stage $t$ problem is
$$
\mathfrak{Q}_t(x_{t-1}, \xi_t)=
\left\{ 
\begin{array}{l}
\inf c_t( y_t - x_{t-1} ) + b_t (\xi_t - y_t)_{+} + h_t (y_t - \xi_t)_{+} + \mathcal{Q}_{t+1}(x_{t})\\
x_{t} = y_t -\xi_t, y_t \geq x_{t-1},
\end{array}
\right.
$$
where $c_t$ is the unit buying cost, $h_t$ is the holding cost, and $b_t$ the backorder cost.
The optimal mean cost is $\mathcal{Q}_1( x_0 )$ where $x_0$ is the initial stock.

In what follows, we present the results of numerical simulations obtained solving this problem with
{\tt{SDDP}}, {\tt{SDDP CS 1}}, {\tt{SDDP CS 2}}, {\tt{MuDA}}, {\tt{CuSMuDA CS 1}}, and 
{\tt{CuSMuDA CS 2}} (we used the same notation as before to denote the solution methods).

\subsubsection{Numerical results}

We consider six values for the number of stages $T$ ($T \in \{5, 10, 15, 20, 25, 30\}$) and for fixed $T$, the following values of the
problem and algorithm parameters are taken:
\begin{itemize} 
\item $c_t=1.5+\cos(\frac{\pi t}{6}), b_t=2.8, h_t=0.2, M_t=M$ for all $t=1,\ldots,T$,
\item $x_0 =10$, $p_{t i}=\frac{1}{M}=\frac{1}{20}$ for all $t,i$,
\item $\xi_1= {\overline{\xi_1}}$ and $(\xi_{t 1},\xi_{t 2},\ldots,\xi_{t M})$ 
corresponds to a sample from the distribution of ${\overline{\xi_t}}(1+ 0.1\varepsilon_t)$ for i.i.d $\varepsilon_t \sim \mathcal{N}(0, 1)$
for $t=2,\ldots,T$, where ${\overline{\xi_t}}=5+0.5t$,
\item for the stopping criterion, in \eqref{formulazsup} $\alpha=0.025$
and $N=S=200$, and $\varepsilon=0.05$ in \eqref{stoppingcriterion}.\\
\end{itemize}

\par {\textbf{Checking the implementations.}}
As for the portfolio problem, for each value of $T \in \{5, 10, 15, 20, 25, 30\}$, we 
checked that upper and lower bounds computed by all 6 methods 
are very close for the three algorithms at the final iteration.\\

\par {\textbf{Computational time and proportion of cuts selected.}} 
Table \ref{tablerunningtime3} shows the CPU time needed to solve our six instances
of inventory management problems. For SDDP, all variants with cut selection yield important
reduction in CPU time and {\tt{SDDP CS 2}}, that uses LML 1, is by far the quickest
on 5 instances. For MuDA, Level 1 is not efficient (CPU time with {\tt{CuSMuDA CS 1}} is much larger than 
CPU time with  {\tt{MuDA}}) whereas LML 1 allows us to drastically reduce
CPU time. As for instances of the portfolio problem of type (P3), the fact that {\tt{CuSMuDA CS 1}}
is much slower than both {\tt{CuSMuDA CS 2}} and {\tt{MuDA}} is that it requires a similar number
of iterations and 
selects much more cuts
than {\tt{CuSMuDA CS 2}} which selects very few cuts for all stages, as can be seen on Figure \ref{fig:f_6} which represents the mean proportion of cuts selected
for {\tt{CuSMuDA CS 1}} and {\tt{CuSMuDA CS 2}} as a function of the number of stages for two instances.

\begin{table}
\centering
\begin{tabular}{|c||c|c|c|c|c|c|}
\hline
 & {\tt{SDDP}}  & {\tt{SDDP CS 1}}   &  {\tt{SDDP CS 2}} & {\tt{MuDA}} & {\tt{CuSMuDA CS 1}}  & {\tt{CuSMuDA CS 2}} \\
\hline
$T=5$ &  0.57   &  0.48  &  0.18  &  1.57  &  2.90  &  0.22   \\
\hline
$T=10$ & 5.9 &  1.22  & 1.42   &26.33 &   43.07   & 1.68  \\
\hline
$T=15$ &21.9  & 20.4   &  4.56  &106.52  &  178.08    & 8.07  \\
\hline
$T=20$ &28.6  &  31.2  &  9.94  & 158.10  &  245.24    & 19.43 \\
\hline
$T=25$ & 36.5& 62.8   &  9.61  & 189.27   &  405.58   & 22.64 \\
\hline
$T=30$ & 77.7 &  63.6  &  18.30  &363.86   & 575.53    & 56.62 \\
\hline
\end{tabular}
\caption{Computational time (in minutes) for solving instances of the inventory  problem of Section \ref{sec:inventory} 
for $M=20$ with {\tt{SDDP}}, {\tt{SDDP CS 1}}, {\tt{SDDP CS 2}}, {\tt{MuDA}}, {\tt{CuSMuDA CS 1}}, and {\tt{CuSMuDA CS 2}}.}
\label{tablerunningtime3}
\end{table}

\begin{figure}
%$$
\centering
\begin{tabular}{cc}
\includegraphics[scale=0.6]{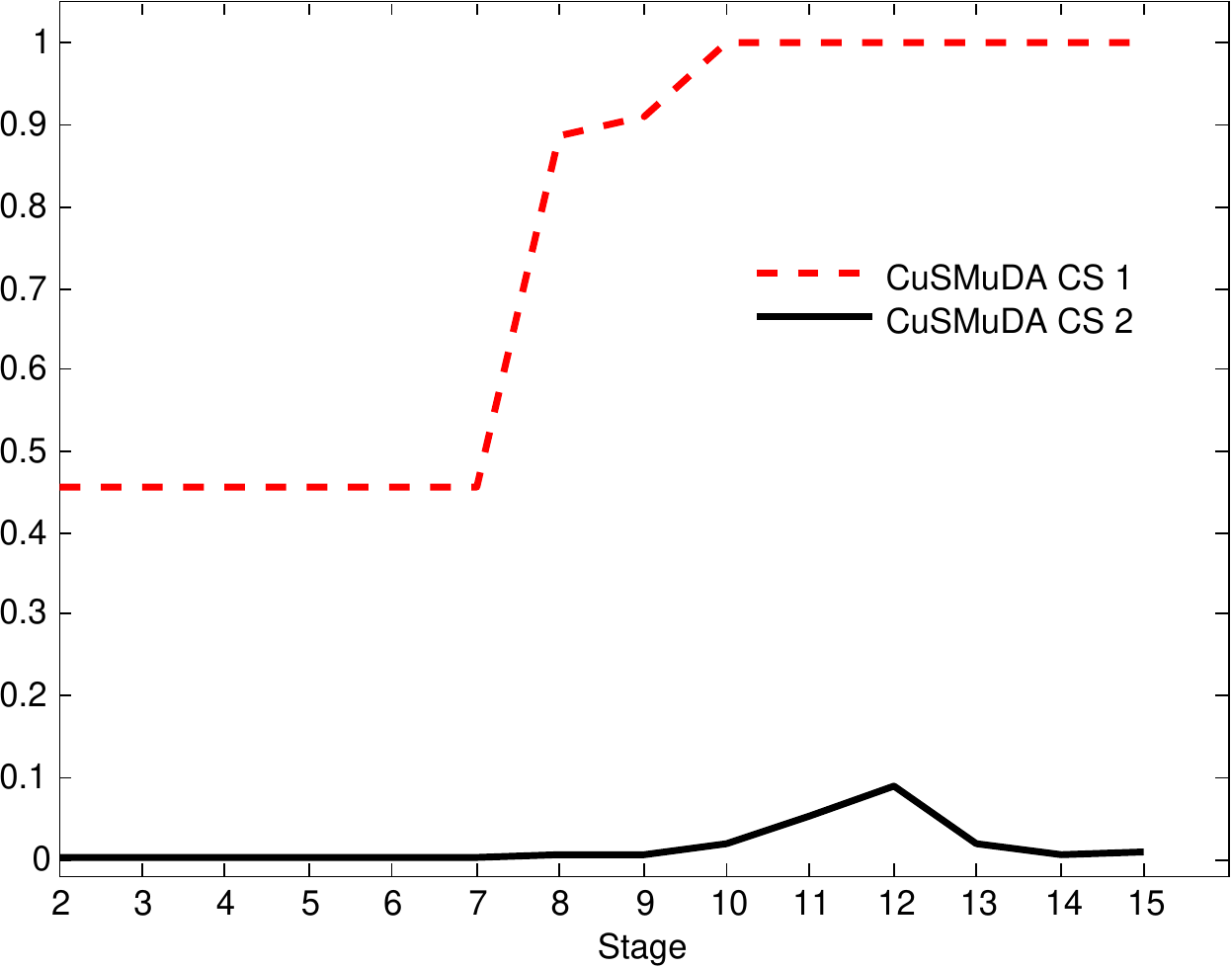}& \includegraphics[scale=0.56]{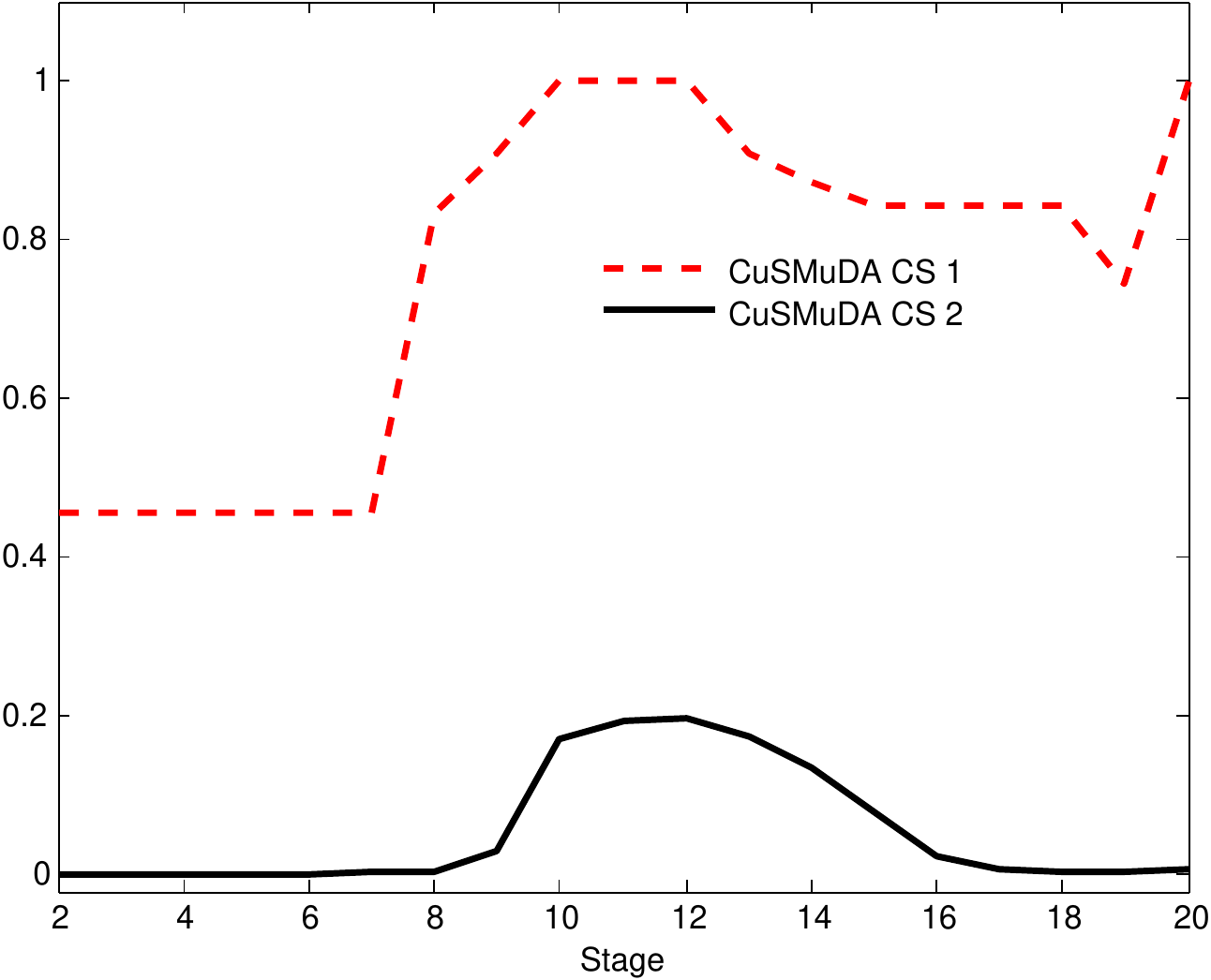} \\
{$T=15, M=20$} & {$T=20, M=20$} 
\end{tabular}
%$$
\caption{ \label{fig:f_6} Mean proportion of cuts (over the iterations of the algorithm) selected for stages $t=2,\ldots,T$ for 
{\tt{CuSMuDA CS 1}} and {\tt{CuSMuDA CS 2}} for the inventory problem for $M=20$.}
\end{figure}

\section{Conclusion}

We proposed CuSMuDA, a  combination of a class of cut selection strategies
with Multicut Decomposition algorithms to solve multistage stochastic linear programs.
We proved the almost sure convergence of the method in a finite number of iterations
and obtained as a by-product the  almost sure convergence in a finite number of iterations
of SDDP combined with this class
of cut selection strategies.

Numerical experiments on many instances of a portfolio and of an inventory problem have shown
that combining LML 1 cut selection with SDDP (resp. MuDA) 
allows us in general to reduce considerably the CPU time of SDDP (resp. MuDA).
There are, however, situations where a variant with selection is slower than its counterpart without cut selection.

It would be interesting to test CuSMuDA and the Limited Memory variant of Level 1 (both for MuDA and SDDP) on other types of stochastic programs and 
to extend the analysis
to nonlinear stochastic programs.

\section*{Acknowledgments} The second author's research was 
partially supported by an FGV grant, CNPq grants 307287/2013-0 and 401371/2014-0, and FAPERJ grant E-26/201.599/2014.
The authors wish to thank Vincent Lecl\`ere for helpful discussions.

\nocite{*}

\addcontentsline{toc}{section}{References}
\bibliographystyle{plain}
\bibliography{Bibliography}

\section*{Appendix}

\begin{figure}[H]
\begin{tabular}{|c|c|}
 \hline 
Level 1 & Limited Memory Level 1 \\
\hline
\begin{tabular}{l}
$I_{t j}^{k}=\{k\}$, $m_{t  j}^k =\mathcal{C}_{t j}^k ( x_{t-1}^k )$.\\
{\textbf{For}} $\ell=1,\ldots,k-1$,\\
\hspace*{0.3cm}{\textbf{If }}$\mathcal{C}_{t j}^k( x_{t-1}^{\ell} ) > m_{t j}^{\ell} + \varepsilon_0 \max(1, |m_{t j}^{\ell}| )$ \\
\hspace*{0.6cm}$I_{t j}^{\ell}=\{k\},\; m_{t j}^{\ell}=\mathcal{C}_{t j}^k( x_{t-1}^{\ell} )$\\
\hspace*{0.3cm}{\textbf{Else if}} $|\mathcal{C}_{t j}^k( x_{t-1}^{\ell} ) - m_{t j}^{\ell}| \leq \varepsilon_0 \max(1, |m_{t j}^{\ell}| )$ \\
\hspace*{0.6cm}$I_{t j}^{\ell}=I_{t j}^{\ell} \cup \{k\}$\\
\hspace*{0.3cm}{\textbf{End If}}\\
\hspace*{0.3cm}{\textbf{If }}$\mathcal{C}_{t j}^{\ell}( x_{t-1}^{k} ) > m_{t j}^k + \varepsilon_0 \max(1, |m_{t j}^{k}| )$ \\
\hspace*{0.6cm}$I_{t j}^k =\{\ell\},\; m_{t j}^k = \mathcal{C}_{t j}^{\ell}( x_{t-1}^{k} )$\\
\hspace*{0.3cm}{\textbf{Else if}} $|\mathcal{C}_{t j}^{\ell}( x_{t-1}^{k} ) - m_{t j}^{k}| \leq \varepsilon_0 \max(1, |m_{t j}^{k}| )$ \\
\hspace*{0.6cm}$I_{t j}^{k}=I_{t j}^{k} \cup \{\ell\}$\\
\hspace*{0.3cm}{\textbf{End If}}\\
{\textbf{End For}}\\
\end{tabular}
&
\begin{tabular}{l}
\vspace*{-0.42cm}\\
$I_{t j}^{k}=\{1\}$, $m_{t  j}^k =\mathcal{C}_{t j}^1 ( x_{t-1}^k )$.\\
{\textbf{For}} $\ell=1,\ldots,k-1$,\\
\hspace*{0.3cm}{\textbf{If }}$\mathcal{C}_{t j}^k( x_{t-1}^{\ell} ) > m_{t j}^{\ell} + \varepsilon_0 \max(1, |m_{t j}^{\ell}| )$ \\
\hspace*{0.6cm}$I_{t j}^{\ell}=\{k\},\; m_{t j}^{\ell}=\mathcal{C}_{t j}^k( x_{t-1}^{\ell} )$\\
\hspace*{0.3cm}{\textbf{End If}}\\
\vspace*{0.4cm}\\
\hspace*{0.3cm}{\textbf{If }}$\mathcal{C}_{t j}^{\ell +1}( x_{t-1}^{k} ) > m_{t j}^k + \varepsilon_0 \max(1, |m_{t j}^{k}| )$ \\
\hspace*{0.6cm}$I_{t j}^k =\{\ell +1\},\; m_{t j}^k = \mathcal{C}_{t j}^{\ell + 1}( x_{t-1}^{k} )$\\
\hspace*{0.3cm}{\textbf{End If}}\\
\vspace*{0.4cm}\\
{\textbf{End For}}\\
\end{tabular}
\\
\hline
\end{tabular}
\caption{Pseudo-codes for selecting the cuts using Level 1 and 
Limited Memory Level 1 taking into account approximation errors.}
\label{figurecut2}
\end{figure}

\end{document}